\documentclass[12pt]{article}
    \usepackage[margin=1in]{geometry}
    \usepackage[utf8]{inputenc}
    \usepackage[T1]{fontenc}    
    \usepackage{amsmath,amssymb,amsfonts,amsthm}
    \usepackage{mathtools}
    \usepackage[symbol]{footmisc}
    \usepackage{url}
    \usepackage{bigints}
    \usepackage{graphicx}
    \usepackage{algorithm2e}
    \usepackage{tabularx}
    \usepackage{xcolor}
    \usepackage{etoc}
    \usepackage{chngcntr}    
\newtheorem{definition}{Defn.}
\newtheorem{cor}{Cor.}
\newtheorem{theo}{Theo.}
\newtheorem{prop}{Prop.}

\RestyleAlgo{ruled}

\begin{document}

\begin{titlepage}
\begin{center}

    \Large\textbf{Framing structural identifiability in terms of parameter symmetries}
       
    \normalsize
        
    \vspace{0.25cm}
    \setcounter{footnote}{0}
    \setlength{\footnotemargin}{0.8em}
    {\normalsize Johannes G Borgqvist\footnote{Corresponding author: \url{johborgq@chalmers.se}}\footnote{Mathematical Sciences, Chalmers University of Technology, Gothenburg, Sweden}, Alexander P Browning\footnote{Mathematical Institute, University of Oxford, United Kingdom\label{Oxford}}\footnote{School of Mathematics and Statistics, University of Melbourne, Melbourne, Australia\label{Melbourne}}, Fredrik Ohlsson\footnote{Department of Mathematics and Mathematical Statistics, Umeå University, Umeå, Sweden}, Ruth E Baker\footref{Oxford}}
    \setlength{\footnotemargin}{1.8em}
        
\abstract{ A key step in mechanistic modelling of dynamical systems is to conduct a structural identifiability analysis. This entails deducing which parameter combinations can be estimated from a given set of observed outputs. The standard differential algebra approach answers this question by re-writing the model as a higher-order system of ordinary differential equations that depends solely on the observed outputs. Over the last decades, alternative approaches for analysing structural identifiability based on Lie symmetries acting on independent and dependent variables as well as parameters, have been proposed. However, the link between the standard differential algebra approach and that using full symmetries remains elusive. In this work, we establish this link by introducing the notion of parameter symmetries, which are a special type of full symmetry that alter parameters while preserving the observed outputs. Our main result states that a parameter combination is locally structurally identifiable if and only if it is a differential invariant of all parameter symmetries of a given model. We show that the standard differential algebra approach is consistent with the concept of structural identifiability in terms of parameter symmetries. We present an alternative symmetry-based approach for analysing structural identifiability using parameter symmetries. Lastly, we demonstrate our approach on two well-known models in mathematical biology. }    
    
    \textbf{Keywords:}\\ Local structural identifiability, parameter symmetries, universal parameter invariants, differential algebra approach. \\      
\end{center}
\end{titlepage}


\section{Introduction}

Given the abundance of experimental data, a substantial amount of effort in mechanistic modelling in biology goes into considering model validation and experimental design. A crucial preliminary step, prior to data collection, is to conduct a \emph{structural identifiability analysis} to determine which model parameters can, in principle, be estimated from a given experimental setup. Structural identifiability is a theoretical property of a mechanistic model that can be investigated analytically under the assumption of perfect, noise-free observations of the model output. In this context, structural identifiability addresses whether parameters can be uniquely inferred from the observed output. A model parameter is said to be globally structurally identifiable if it can be uniquely determined, and locally structurally identifiable if it admits only finitely many distinct values consistent with the observed output.

For models formulated as systems of ordinary differential equations (ODEs), a wide range of methods for assessing structural identifiability is available (see~\cite{chis2011structural} for a review). The standard approach for establishing global structural identifiability of ODE models with rational right-hand sides is based on differential algebra~\cite{ljung1994global,hong2020global,saccomani2001new,walter1982global}. This approach reformulates the original system as an equivalent input–output representation comprising polynomial differential equations, potentially of higher order, involving only the observed inputs, outputs, and model parameters. Structural identifiability is then assessed by constructing a map between the original model parameters and the parameter combinations appearing in the input–output equations, and determining whether this map is injective. Identifiable parameter quantities can subsequently be recovered by exploiting the polynomial structure of the input–output system, typically through analysis of the corresponding coefficients.

To illustrate the differential algebra approach, let $u(t)$ and $v(t)$ denote the concentrations of two chemical species that depend on time, $t$, and which satisfy two decoupled ODEs
\begin{equation}
  \dfrac{\mathrm{d}u}{\mathrm{d}t}=\kappa_{1}-\lambda{u},\quad\dfrac{\mathrm{d}v}{\mathrm{d}t}=\kappa_{2}-\lambda{v},
  \label{eq:toy_individual}
\end{equation}
where both species decay at rate $\lambda>0$ and are synthesised at rates rates $\kappa_{1}>0$ and $\kappa_{2}>0$, respectively. We assume that we observe only the total concentration $y(t)=u(t)+v(t)$, yielding the following model for the output $y$:
\begin{equation}
  \dfrac{\mathrm{d}y}{\mathrm{d}t}+\lambda{y}-(\kappa_{1}+\kappa_{2})=0.
  \label{eq:toy_intro}        
\end{equation}
Next, we extract the coefficients in front of $\{\mathrm{d}y/\mathrm{d}t,y,1\}$ resulting in the set $\{1,\lambda,\kappa_{1}+\kappa_{2}\}$. Clearly, the set of identifiable parameter quantities is given by  $\{\lambda,\kappa_{1}+\kappa_{2}\}$ implying that the decay rate $\lambda$ is (globally) structurally identifiable. On the other hand, the synthesis rates $\kappa_{1}$ and $\kappa_{2}$ are individually non-identifiable whereas the sum $\kappa_{1}+\kappa_{2}$ is globally structurally identifiable.

In addition to the standard differential algebra approach, other methodologies for analysing the structural identifiability of systems of first-order ODEs have also been developed. In particular, for more than two decades, methods for analysing structural identifiability using a special type of \emph{Lie point symmetries} we refer to as \emph{full symmetries} have been developed~\cite{yates2009structural,merkt2015higher,massonis2020finding,castro2020structuralIdentifiability,villaverde2022symmetries}. Full symmetries are transformations called $C^{\infty}$ diffeomorphisms which map solutions of the system of ODEs of interest to other solutions while simultaneously preserving the observed outputs. In particular, they act on the independent variable (corresponding to time in our examples), the dependent variables (corresponding to the states), the outputs and the parameters.

An example of a class of full symmetries studied by Castro and de Boer~\cite{castro2020structuralIdentifiability} is scalings, in which state variables and parameters are transformed multiplicatively by scaling factors that leave the model equations invariant. Within this framework, Castro and de Boer analysed both the structural identifiability of parameters and the related notion of structural observability of states, defined as the ability to reconstruct state variables from observed outputs. In the presence of such scaling symmetries, parameters are structurally identifiable and states are structurally observable if and only if they are fixed points of the induced scaling action; equivalently, the only scaling transformation that preserves the observed outputs is the identity transformation, for which all scaling factors are equal to one. Critically, not all models possess scalings as full symmetries, and there are models with other types of full symmetries that can also be used as a basis for deducing structural identifiability~\cite{villaverde2021testing}. Hence, relying solely on scaling symmetries in the context of structural identifiability can be misleading~\cite{villaverde2021testing}.

Importantly, full symmetries generalise the classical notion of symmetries of ordinary differential equations, which act on the independent and dependent variables and map solutions to solutions, but do not act on model parameters or observed outputs. Classical symmetry methods are primarily used to construct analytical solutions, reduce the order of ODEs, or classify families of differential equations with prescribed symmetry properties~\cite{bluman1989symmetries,hydon2000symmetry,olver2000applications,stephani1989differential}. A well-known limitation of classical symmetry analysis for first-order ODE systems is that their symmetry groups are infinite-dimensional. As a consequence, the associated linearised symmetry conditions—whose solutions define the infinitesimal generators of the symmetry group—are underdetermined, with more unknown infinitesimals than equations. In practice, this necessitates the introduction of restrictive ans\"{a}tze in order to obtain tractable solutions. This difficulty is exacerbated in the context of structural identifiability analysis, where full symmetries introduce additional infinitesimals associated with model parameters, further increasing the degree of underdetermination. As a result, recent work on symmetry-based identifiability analysis has focused on the automated computation of full symmetries~\cite{merkt2015higher,massonis2020finding}. These approaches impose polynomial ans\"{a}tze on the infinitesimals—taken to be multivariate polynomials in the independent variables, state variables, and parameters—thereby reducing the linearised symmetry conditions to a finite-dimensional linear system that can be solved algorithmically, for example using Gaussian elimination. Despite these advances, fundamental questions remain unresolved, notably the precise relationship between full symmetries and structural identifiability, and the extent to which symmetry-based approaches are consistent with the classical differential algebra framework. Resolving these questions is essential if the full machinery of symmetry methods is to be systematically exploited for identifiability analysis in mechanistic models.

In this work, we establish the link between structural identifiability and full symmetries by introducing the notion of \textit{parameter symmetries}. These are full symmetries solely acting on parameters; that is, Lie symmetries acting as re-parametrisations of the model of interest. Provided the notion of parameter symmetries, we show that a parameter is locally structurally identifiable if and only if it is a so-called \textit{universal parameter invariant}---a differential invariant of all parameter symmetries of a given model. Next, we demonstrate that the standard differential algebra approach for deducing global structural identifiability will always find universal parameter invariants and thus it is consistent with the notion of parameter symmetries. Subsequently, we develop an alternative methodology for deducing local structural identifiability based on parameter symmetries (Algorithm \ref{alg:CaLinInv}). Specifically, the method: (i) re-writes the original system of ODEs in terms of the observed outputs which are called the \emph{canonical coordinates}; (ii) finds parameter symmetries by solving the linearised symmetry conditions; and (iii) calculates universal parameter invariants. Our approach finds both the locally structurally identifiable parameter quantities as well as the parameter transformations preserving these parameter quantities, namely the parameter symmetries. Lastly, we conduct local structural identifiability analyses of a glucose-insulin model and an epidemiological SEI model using Algorithm \ref{alg:CaLinInv}, yielding insights about identifiable parameter combinations as well as the family of parameter transformations which leaves them invariant. Overall, we demonstrate how local structural identifiability of mechanistic models can be understood in terms of parameter symmetries. 

For the interested reader, we present definitions of global and local structural identifiability as well as the standard differential algebra approach in Section 2. Thereafter, we present the theoretical link between our parameter symmetries and local structural identifiability in Section 3. For the reader who is interested in the application of Lie symmetries in local structural identifiability analyses of models consisting of systems of ODEs, proceeding to application of parameter symmetries in local structural identifiability analyses of specific models in biology in Section 4 is sufficient.


\section{Definitions of structural identifiability}

We briefly present definitions of both global and local structural identifiability followed by an algorithm for the standard differential algebra approach. To this end, consider the following \emph{state-output system} of first-order ODEs and associated outputs
\begin{equation}
\begin{split}
    \dfrac{\mathrm{d}\mathbf{x}}{\mathrm{d}t}&=\mathbf{f}\left(t,\mathbf{x},\boldsymbol{\theta}\right),\\
    \mathbf{y}&=\mathbf{h}\left(t,\mathbf{x},\boldsymbol{\theta}\right).
    \end{split}
    \label{eq:ODE_sys_full}
\end{equation}
Here, $t\in\mathbb{R}$ is the independent variable corresponding to time, $\boldsymbol{\theta}\in\mathbb{R}^{p}$ is the vector of $p$ parameters, $\mathbf{x}(t,\boldsymbol{\theta})\in\mathbb{R}^{n}$ are the $n\geq1$ dependent variables corresponding to the states of the system and $\mathbf{y}(t,\mathbf{x},\boldsymbol{\theta})\in\mathbb{R}^{m}$ are the $1\leq{m}\leq n$ observed outputs. In the context of structural identifiability, we ask ourselves which parameters can be determined based on a set of $1\leq m\leq n$ observed outputs, $\mathbf{y}$, given $n$ states, $\mathbf{x}$? Here, we make two important assumptions, namely that the functions $\mathbf{f}$ and $\mathbf{h}$ are analytical and rational functions of the independent and dependent variables and that they are infinitely differentiable. Under these assumptions, we can always reduce the original state-output system of first-order ODEs to an \textit{output system} corresponding to a (potentially higher order) polynomial system of ODEs solely depending on the outputs~\cite{ljung1994global}
\begin{equation}
    \Delta\left(t,\frac{\mathrm{d}^{N}\mathbf{y}}{\mathrm{d}t^{N}},\frac{\mathrm{d}^{N-1}\mathbf{y}}{\mathrm{d}t^{N-1}},\ldots,\dfrac{\mathrm{d}^{2}\mathbf{y}}{\mathrm{d}t^{2}},\dfrac{\mathrm{d}\mathbf{y}}{\mathrm{d}t},\mathbf{y},\boldsymbol{\theta}\right)=\mathbf{0},\\
    \label{eq:ODE_sys_output}
\end{equation}
for some power $N\in\mathbb{N}_{+}$ and where $\Delta$ is a vector-valued function of multivariate polynomials. The output system in Eq. \eqref{eq:ODE_sys_output} is output-equivalent to the original state-output system in Eq. \eqref{eq:ODE_sys_full} meaning that any outputs $\mathbf{y}$ that are solutions of the latter system are also solutions of the former system and \textit{vice versa}. The standard differential algebra approach for elucidating structural identifiability analyses the output system, and thus requires that the functions $\mathbf{f}$ and $\mathbf{h}$ defining the state-output system be rational functions of the states. Unlike the standard differential algebra approach, the approach for analysing structural identifiability based on Lie symmetries does not depend on these assumptions. However, since the the aim of this work is to establish the link between these two approaches we restrict ourselves to rational functions $\mathbf{f}$ and $\mathbf{h}$. Provided this problem formulation, we next present the definition of global and local structural identifiability as well as the standard differential algebra approach for conducting a global structural identifiability analysis. 


\subsection{Global and local structural identifiability}

We first present the definition of global structural identifiability~\cite{renardy2022structural,meshkat2015identifiability,meshkat2014repara}. 

\medskip
\begin{definition}[Global structural identifiability of parameters]
  An individual parameter $\theta_{j}\in\boldsymbol{\theta}$, ${j}\in{1,\ldots,p}$ is globally \textit{structurally identifiable} if for almost every value $\boldsymbol{\theta}^{\star}$ and for almost all initial conditions the following holds:
  \begin{equation}
    \mathbf{y}(t,\boldsymbol{\theta})=\mathbf{y}(t,\boldsymbol{\theta}^{\star})\quad\forall{t}\in\mathbb{R}\implies{\theta}_{j}=\theta_{j}^{\star}\quad\quad\forall\,\boldsymbol{\theta},\boldsymbol{\theta}^{\star}\in\mathbb{R}^{p}.
    \label{eq:structural identifiability}
  \end{equation}
\label{def:structural identifiability}
\end{definition}
The idea is essentially that a parameter is globally structurally identifiable if any change in the parameter results in a change in the output. A closely related concept is that of local structural identifiability which means that a local change in parameters results in a change in the output. Here, the effect of changing parameters on observed outputs is analysed in a ball of radius $\rho>0$ centred around a given parameter vector $\boldsymbol{\theta}_{C}\in\mathbb{R}^{p}$, and we denote this ball by $\mathcal{B}_{\rho}\left(\boldsymbol{\theta}_{C}\right)\coloneqq\left\{\boldsymbol{\theta}\in\mathbb{R}^{p}:|\boldsymbol{\theta}-\boldsymbol{\theta}_{C}|<\rho\right\}$.

\begin{definition}[Local structural identifiability of parameters]
  An individual parameter $\theta_{j}\in\boldsymbol{\theta}$, ${j}\in{1,\ldots,p}$ is locally \textit{structurally identifiable} if for almost every value $\boldsymbol{\theta}^{\star}$ and almost all initial conditions the following holds:
  \begin{equation}
    \mathbf{y}(t,\boldsymbol{\theta})=\mathbf{y}(t,\boldsymbol{\theta}^{\star})\quad\forall{t}\in\mathbb{R}\implies{\theta}_{j}=\theta_{j}^{\star}\quad\quad\forall\,\boldsymbol{\theta},\boldsymbol{\theta}^{\star}\in\mathcal{B}_{\rho}\left(\boldsymbol{\theta}_{C}\right),
    \label{eq:local_structural identifiability}
  \end{equation}
  for some radius $\rho>0$ and some centre $\boldsymbol{\theta}_{C}\in\mathbb{R}^p$. 
\label{def:local_structural identifiability}
\end{definition}

A globally structurally identifiable parameter is always locally structurally identifiable, but the converse statement is not always true. Similarly, we can also define globally identifiable parameter combinations~\cite{meshkat2015identifiability,meshkat2014repara}, such as the sum of two parameters, in terms of functions $f:\mathbb{R}^{p}\mapsto{\mathbb{R}}$ of the parameters (Defn. \ref{def:global_structural identifiability_comb})

\begin{definition}[Global structural identifiability of parameter combinations]
  The function $f:\mathbb{R}^{p}\mapsto{\mathbb{R}}$, where $f(\boldsymbol{\theta})$ defines a parameter combination, is globally \textit{structurally identifiable} if for almost every value $\boldsymbol{\theta}^{\star}$ and almost all initial conditions the following holds:
  \begin{equation}
    \mathbf{y}(t,\boldsymbol{\theta})=\mathbf{y}(t,\boldsymbol{\theta}^{\star})\quad\forall{t}\in\mathbb{R}\implies{f}(\boldsymbol{\theta})=f(\boldsymbol{\theta}^{\star})\quad\quad\forall\,\boldsymbol{\theta},\boldsymbol{\theta}^{\star}\in\mathbb{R}^{p}\,.
    \label{eq:global_structural identifiability_comb}
  \end{equation}
\label{def:global_structural identifiability_comb}
\end{definition}

Locally identifiable parameter combinations are defined analogously~\cite{meshkat2015identifiability,meshkat2014repara}. Subsequently, we present the standard differential algebra approach for deducing 
global structural identifiability. 


\subsection{The standard differential algebra approach}

he standard differential algebra approach enables a global structural identifiability analysis for systems of ordinary differential equations whose right-hand sides $\mathbf{f}$ and output functions $\mathbf{h}$ are rational, as in Eq.~\eqref{eq:ODE_sys_full}. For such models, the approach identifies globally structurally identifiable parameter combinations through a sequence of algebraic steps~\cite{renardy2022structural}. First, the original state–output system in Eq.~\eqref{eq:ODE_sys_full} is transformed into an equivalent, potentially higher-order, output-only system of differential equations, as in Eq.~\eqref{eq:ODE_sys_output}. Second, this input–output representation is rewritten as a system of monic polynomial differential equations. Third, the coefficients of these polynomials are extracted. Finally, the globally structurally identifiable parameter combinations are inferred from these coefficients.

The validity of this procedure follows from the output-equivalence of the original state–output system and the derived output system. Specifically, for any choice of parameters $\boldsymbol{\theta}$, the outputs $\mathbf{y}$ generated by a solution $\mathbf{x}$ of the original system in Eq.~\eqref{eq:ODE_sys_full} also satisfy the corresponding output system in Eq.~\eqref{eq:ODE_sys_output}~\cite{audoly2001global}. In this sense, the output system provides an exhaustive characterisation of the input–output behaviour of the original model~\cite{audoly2001global}. Since solutions of ordinary differential equations are uniquely determined by their governing equations and initial conditions, the parameter combinations appearing in the output system uniquely determine the observed outputs $\mathbf{y}$, and hence correspond precisely to the globally structurally identifiable quantities.


\section{Local structural identifiability and parameter symmetries}

To establish the missing link between structural identifiability analyses and Lie symmetries, we first introduce the notion of parameter symmetries. Using this notion, we show that the \emph{universal parameter invariants} of these parameter symmetries correspond to the \emph{locally structurally identifiable parameter quantities} of a given model. Using this result, we demonstrate that the standard differential algebra approach for analysing global structural identifiability is consistent with the concept of universal parameter invariants associated with parameter symmetries.


\subsection{Parameter symmetries are Lie transformations acting as re-parametrisations that preserve observed outputs}

Structural identifiability is defined by the requirement that changes in parameter values induce changes in the observed outputs. It is therefore natural to study transformations of the model parameters that leave the outputs invariant. Among such transformations, Lie symmetries play a distinguished role: they are infinitesimal transformations that act locally on parameters and states while preserving the model equations. Motivated by this observation, we introduce a particular class of Lie symmetries, termed parameter symmetries.

\medskip
\begin{definition}[Parameter symmetries]
    Let $\Gamma^{\boldsymbol{\theta}}_{\varepsilon}:\mathbb{R}\times\mathbb{R}^{m}\times\mathbb{R}^{p}\mapsto\mathbb{R}\times\mathbb{R}^{m}\times\mathbb{R}^{p}$ be a one-parameter $\mathcal{C}^{\infty}$ diffeomorphism that is restricted to the parameters of the system of output ODEs in Eq.~\eqref{eq:ODE_sys_output} defined by
\begin{equation}
    \Gamma_{\varepsilon}^{\boldsymbol{\theta}}:(t,\mathbf{y},\boldsymbol{\theta})\mapsto\left(t,\mathbf{y},\hat{\boldsymbol{\theta}}(\boldsymbol{\theta};\varepsilon)\right),
    \label{eq:Lie_symmetry_parameter}
\end{equation}
where the target functions $\hat{\boldsymbol{\theta}}$ depend solely on the parameters $\boldsymbol{\theta}$ in addition to the parameter $\varepsilon$. To simplify the notation, we simply write $\hat{\boldsymbol{\theta}}(\varepsilon)$ as opposed to $\hat{\boldsymbol{\theta}}(\boldsymbol{\theta};\varepsilon)$. The independent variable $t$ and the dependent variables $\mathbf{y}$ are invariant under the action of $\Gamma_{\varepsilon}^{\boldsymbol{\theta}}$, implying that for any solution curves the following conservation property holds:
\begin{equation}
  \mathbf{y}(t,\boldsymbol{\theta})=\mathbf{y}(t,\hat{\boldsymbol{\theta}}(\varepsilon))\quad\forall\,t,\varepsilon\in\mathbb{R}.
  \label{eq:invariant_output}
\end{equation}
Moreover, let $X_{\boldsymbol{\theta}}$ be the corresponding infinitesimal generator of the Lie group
\begin{equation}
X_{\boldsymbol{\theta}}=\sum_{\ell=1}^{p}\chi_{\ell}(\boldsymbol{\theta})\partial_{\theta_{\ell}},
    \label{eq:X_parameter_1}
\end{equation}
where the functions $\boldsymbol{\chi}(\boldsymbol{\theta})=\left(\chi_{1}(\boldsymbol{\theta}),\ldots,\chi_{p}(\boldsymbol{\theta})\right)$ are called the \textit{parameter infinitesimals}. Then $\Gamma^{\boldsymbol{\theta}}_{\varepsilon}$ in Eq.~\eqref{eq:Lie_symmetry_parameter} is a \textit{parameter symmetry} of the system of output ODEs in Eq.~\eqref{eq:ODE_sys_output} if its infinitesimal generator $X_{\boldsymbol{\theta}}$ solves the linearised symmetry conditions given by
\begin{equation} \left.X_{\boldsymbol{\theta}}\left(\Delta\left(t,\frac{\mathrm{d}^{N}\mathbf{y}}{\mathrm{d}t^{N}},\frac{\mathrm{d}^{N-1}\mathbf{y}}{\mathrm{d}t^{N-1}},\ldots,\frac{\mathrm{d}\mathbf{y}}{\mathrm{d}t},\mathbf{y},\boldsymbol{\theta}\right)\right)\right|_{\Delta=\mathbf{0}}=\sum_{\ell=1}^{p}\chi_{\ell}(\boldsymbol{\theta})\frac{\partial\Delta}{\partial_{\theta_{\ell}}}=\mathbf{0}.
    \label{eq:lin_sym_parameter}
\end{equation}
    \label{def:parameter}
\end{definition}

Essentially, a parameter symmetry $\Gamma_{\varepsilon}^{\boldsymbol{\theta}}$ is a \emph{re-parametrisation} of the output system in Eq.~\eqref{eq:ODE_sys_output} which preserves the observed outputs $\mathbf{y}$. Since these parameter symmetries are restricted to the parameters of the model, we will use the notation $\Gamma_{\varepsilon}^{\boldsymbol{\theta}}:\boldsymbol{\theta}\mapsto\hat{\boldsymbol{\theta}}(\varepsilon)$ to describe them henceforth. Next, we define the notion of differential invariants of parameter symmetries.

\medskip
\begin{definition}[Differential invariants of parameter symmetries]
  Consider a parameter symmetry $\Gamma^{\boldsymbol{\theta}}_{\varepsilon}$ and its corresponding infinitesimal generator $X_{\boldsymbol{\theta}}$. A non-constant function 
  \begin{equation*}I=I\left(t,\frac{\mathrm{d}^{N}\mathbf{y}}{\mathrm{d}t^{N}},\frac{\mathrm{d}^{N-1}\mathbf{y}}{\mathrm{d}t^{N-1}},\ldots,\frac{\mathrm{d}\mathbf{y}}{\mathrm{d}t},\mathbf{y},\boldsymbol{\theta}\right),
  \end{equation*}
  is called a \textit{differential invariant} of $\Gamma^{\boldsymbol{\theta}}_{\varepsilon}$ if it satisfies
  \begin{equation}
X_{\boldsymbol{\theta}}\left(I\right)=\sum_{\ell=1}^{p}\chi_{\ell}(\boldsymbol{\theta})\frac{\partial I}{\partial_{\theta_{\ell}}}=0.
    \label{eq:invariant}
   \end{equation}
  \label{def:inv}
\end{definition}

From this definition, we immediately see that the independent time variable $t$, all outputs $\mathbf{y}$, and their respective derivatives, are themselves differential invariants. In addition to the independent and dependent variables, there are other invariants solely depending on the parameters, and to distinguish between these two types of invariants we introduce the notion of \textit{parameter invariants}.

\medskip
\begin{definition}[Parameter invariants of parameter symmetries]
  Consider a parameter symmetry $\Gamma^{\boldsymbol{\theta}}_{\varepsilon}$ and its corresponding infinitesimal generator $X_{\boldsymbol{\theta}}$. A non-constant function $I_{\boldsymbol{\theta}}=I_{\boldsymbol{\theta}}\left(\boldsymbol{\theta}\right)$ is a called a \textit{parameter invariant} of $\Gamma^{\boldsymbol{\theta}}_{\varepsilon}$ if it solves
  \begin{equation}X_{\boldsymbol{\theta}}\left(I_{\boldsymbol{\theta}}\right)=\sum_{\ell=1}^{p}\chi_{\ell}(\boldsymbol{\theta})\frac{\partial I_{\boldsymbol{\theta}}}{\partial_{\theta_{\ell}}}=0.
    \label{eq:parameter_invariant}
   \end{equation}
  \label{def:param_inv}  
\end{definition}

Moreover, for a general output system of ODEs there are potentially many possible re-parametrisations, or, differently put, many possible parameter symmetries $\Gamma_{\varepsilon}^{\boldsymbol{\theta}}$ (Defn.~\ref{def:parameter}). For instance, consider the trivial parameter symmetry $\Gamma_{\varepsilon}^{\boldsymbol{\theta},0}:\boldsymbol{\theta}\mapsto\boldsymbol{\theta}$ which is common to \emph{all} possible output systems. For the trivial parameter symmetry $\Gamma_{\varepsilon}^{\boldsymbol{\theta},0}$, every single parameter is a parameter invariant, and this is the only parameter symmetry for which this is true. Thus, for an output system with an additional parameter symmetry $\Gamma_{\varepsilon}^{\boldsymbol{\theta}}$ other than the trivial parameter symmetry, some of the parameter invariants of $\Gamma_{\varepsilon}^{\boldsymbol{\theta},0}$ are shared with $\Gamma_{\varepsilon}^{\boldsymbol{\theta}}$ while other parameter invariants of the trivial parameter symmetry are unique to $\Gamma_{\varepsilon}^{\boldsymbol{\theta},0}$. In general, parameter symmetries have some parameter invariants that are unique to them, and other parameter invariants that are shared with all other parameter symmetries. We refer to the latter type of parameter invariant as \textit{universal parameter invariants}.

\medskip
\begin{definition}[Universal parameter invariants of a model]
A non-constant function $I_{\boldsymbol{\theta}}=I_{\boldsymbol{\theta}}\left(\boldsymbol{\theta}\right)$ that is a parameter invariant of all parameter symmetries of the system of output ODEs in Eq.~\eqref{eq:ODE_sys_output} is called a \textit{universal parameter invariant}.
  \label{def:universal}  
\end{definition}

\medskip
Importantly, the independent variable $t$, the outputs $\mathbf{y}$, and all derivatives of the outputs, are universal invariants of all parameter symmetries.

Given the system of output ODEs in Eq.~\eqref{eq:ODE_sys_output}, we know that the number of parameter invariants for any parameter symmetry $\Gamma_{\varepsilon}^{\boldsymbol{\theta}}$ is an integer in the set $\{0,\ldots,p\}$ where $p$ is the number of parameters. When we have zero parameter invariants, there are no parameters in the output ODEs, and in the case of $p$ parameter invariants, then each parameter is itself an invariant---which is only true for the trivial parameter symmetry $\Gamma_{\varepsilon}^{\boldsymbol{\theta},0}:\boldsymbol{\theta}\mapsto\boldsymbol{\theta}$. From now on, we restrict ourselves to output ODEs in the form of Eq.~\eqref{eq:ODE_sys_output} that contain parameters, i.e.~we exclude the extreme case where we have $p=0$ parameters. In particular, if the parameter $\theta_{\ell}$ for some $\ell\in\{1,\ldots,p\}$ is a parameter invariant, i.e.~$I_{\boldsymbol{\theta}}(\boldsymbol{\theta})=\theta_{\ell}$, then Eq.~\eqref{eq:parameter_invariant} gives
\begin{equation}
  X_{\boldsymbol{\theta}}(\theta_{\ell})=\chi_{\ell}(\boldsymbol{\theta})=0.
\label{eq:individual_parameter_invariant}
\end{equation}
Consequently, \textit{a parameter $\theta_{\ell}$ is invariant if its infinitesimal is zero, i.e. $\chi_{\ell}=0$}. Moreover, by generating the corresponding transformation $\hat{\theta}_{\ell}(\varepsilon)$ using the infinitesimal $\chi_{\ell}$, which entails solving the following ODE:
\begin{equation}
   \frac{\mathrm{d}\hat{\theta}_{\ell}}{\mathrm{d}\varepsilon}=\chi_{\ell}\left(\hat{\boldsymbol{\theta}}(\varepsilon)\right)=0,\quad{\hat{\theta}_{\ell}(\varepsilon=0)=\theta_{\ell}},
  \label{eq:param_trans_eq}
\end{equation}
we obtain that the parameter transformation $\hat{\theta}_{\ell}(\varepsilon)$ which leaves the parameter $\theta_{\ell}$ invariant is given by
\begin{equation}
  \hat{\theta}_{\ell}(\varepsilon)=\theta_{\ell}\quad\forall\varepsilon\in\mathbb{R}.
  \label{eq:param_trans}
\end{equation}
In other words, whenever a parameter $\theta_{\ell}$ is a parameter invariant, then it is characterised by Eqs. \eqref{eq:individual_parameter_invariant} and \eqref{eq:param_trans}, implying that it is conserved under transformations by the parameter symmetry $\Gamma_{\varepsilon}^{\boldsymbol{\theta}}$.

Given the notion of parameter invariants, we proceed by formulating the notion of local structural identifiability in terms of parameter symmetries. 


\subsection{Local structural identifiability defined in terms of universal parameter invariants}

One of the key properties of pointwise Lie symmetries of one parameter, in general, and parameter symmetries, in particular, is that these transformations are defined by single-variable, real-valued analytic functions. Based on this property, we now present our main result expressing structural identifiability in terms of universal parameter invariants.

\begin{theo}[Local structural identifiability of individual parameters in terms of universal parameter invariants]
 Let $\mathbf{y}(t,\boldsymbol{\theta})$ be a solution of the system of output ODEs in Eq.~\eqref{eq:ODE_sys_output}. A parameter $\theta_{\ell}\in\boldsymbol{\theta}$, ${\ell}\in\{1,\ldots,p\}$ is locally structurally identifiable if and only if it is a universal parameter invariant.
\label{thm:structural identifiability_symmetries}
\end{theo}
\begin{proof}
  ``$\Longrightarrow$'' We need to prove that a locally structurally identifiable parameter $\theta_{\ell}$ is also a universal parameter invariant. To this end, take a radius $\rho>0$ and a centre $\boldsymbol{\theta}_{C}\in\mathbb{R}^{p}$ defining a ball $\mathcal{B}_{\rho}\left(\boldsymbol{\theta}_{C}\right)$ in which the parameters are identifiable, which means that the following holds:
  \begin{equation}
  \mathbf{y}(t,\boldsymbol{\theta})=\mathbf{y}(t,\boldsymbol{\theta}^{\star})\quad\forall{t}\in\mathbb{R}\implies{\theta}_{\ell}=\theta_{\ell}^{\star}\quad\quad\forall\boldsymbol{\theta},\boldsymbol{\theta}^{\star}\in\mathcal{B}_{\rho}\left(\boldsymbol{\theta}_{C}\right)\,.
      \label{eq:proof_eq_1}
  \end{equation}
  Take a parameter symmetry $\Gamma_{\varepsilon}^{\boldsymbol{\theta}}$. Since this parameter symmetry acts infinitesimally and pointwise, starting from any parameter $\boldsymbol{\theta}\in\mathcal{B}_{\rho}(\boldsymbol{\theta}_{C})$ we can always find transformation parameters $\varepsilon_{1},\varepsilon_{2}\in\mathbb{R}$ where $\varepsilon_{1}<\varepsilon_{2}$ such that these parameters form an interval $[\varepsilon_{1},\varepsilon_{2}]$ for which the corresponding transformed parameters $\hat{\boldsymbol{\theta}}(\varepsilon)$ lie in the ball where the parameters are locally identifiable, which means that $\hat{\boldsymbol{\theta}}(\varepsilon)\in\mathcal{B}_{\rho}\left(\boldsymbol{\theta}_{C}\right)$ for all $\varepsilon\in[\varepsilon_{1},\varepsilon_{2}]$. Also, the independent variable (i.e. time, $t$) and the dependent variables (i.e. the outputs, $\mathbf{y}$) are invariant under transformations by the parameter symmetry $\Gamma_{\varepsilon}^{\boldsymbol{\theta}}$, and thus the following conservation property holds as well: $\mathbf{y}(t,\boldsymbol{\theta})=\mathbf{y}(t,\hat{\boldsymbol{\theta}}(\varepsilon))$ for all ${t,\varepsilon}\in\mathbb{R}$.
  Consequently, it follows from Eq.~\eqref{eq:proof_eq_1} that
  $$ \mathbf{y}(t,\boldsymbol{\theta})=\mathbf{y}(t,\hat{\boldsymbol{\theta}}(\varepsilon))\quad\forall{t}\in\mathbb{R}\implies{\theta}_{\ell}=\hat{\theta}_{\ell}(\varepsilon)\quad\quad\forall\varepsilon\in[\varepsilon_{1},\varepsilon_{2}]\,,$$
  whenever $\theta_{\ell}$ is locally structurally identifiable. This means that $\hat{\theta}_{\ell}(\varepsilon)$ is constant in the subdomain $[\varepsilon_{1},\varepsilon_{2}]$. Since $\hat{\theta}_{\ell}(\varepsilon)$ is a single variable real-valued analytic function, it must therefore be constant throughout its entire domain, which means that $\hat{\theta}_{\ell}(\varepsilon)={\theta}_{\ell}\,\,\forall\varepsilon\in\mathbb{R}$, and thus $\theta_{\ell}$ is a parameter invariant of the parameter symmetry $\Gamma_{\varepsilon}^{\boldsymbol{\theta}}$. The same arguments hold for all parameter symmetries, and thus $\theta_{\ell}$ is a universal parameter invariant. \medskip \\
  ``$\Longleftarrow$'' We need to prove that if $\theta_{\ell}$ is a universal parameter invariant, then it is also locally structurally identifiable. For a universal parameter invariant, we have that $\hat{\theta}_{\ell}(\varepsilon)=\theta_{\ell}\,\,\forall\varepsilon\in\mathbb{R}$ holds for all parameter symmetries. Since all parameter symmetries leave both the independent and dependent variables invariant, the conservation property $\mathbf{y}(t,\boldsymbol{\theta})=\mathbf{y}(t,\hat{\boldsymbol{\theta}}(\varepsilon))$ for all ${t,\varepsilon}\in\mathbb{R}$ holds as well for all parameter symmetries. Together, these two properties yield that
  \begin{equation}
  \mathbf{y}(t,\boldsymbol{\theta})=\mathbf{y}(t,\hat{\boldsymbol{\theta}}(\varepsilon))\quad\forall{t}\in\mathbb{R}\quad\implies\quad\theta_{\ell}=\hat{\theta}_{\ell}(\varepsilon)\quad\forall\varepsilon\in\mathbb{R},
  \label{eq:proof_eq_2}
  \end{equation}
  holds for all parameter symmetries whenever $\theta_{\ell}$ is a universal parameter invariant. Take a radius $\rho>0$ and form the ball $\mathcal{B}_{\rho}\left(\boldsymbol{\theta}\right)$. We need to prove that there exists a radius $\rho$ such that the parameter symmetries for which the implication in Eq.~\eqref{eq:proof_eq_2} holds also cover the whole ball $\mathcal{B}_{\rho}\left(\boldsymbol{\theta}\right)$. By ``covering the whole ball'', we mean that any $\boldsymbol{\theta}^{\star} \in \mathcal{B}_{\rho}(\boldsymbol{\theta})$ yielding identical outputs for all $t \in \mathbb{R}$ must be related to $\boldsymbol{\theta}$ through a parameter symmetry, i.e., $\boldsymbol{\theta}^{\star} = \hat{\boldsymbol{\theta}}(\varepsilon^{\star})$ for some $\varepsilon^{\star} \in \mathbb{R}$. Infinitesimally, the linearised symmetry conditions in Eq.~\eqref{eq:lin_sym_parameter} state that the only directions in parameter space that preserve the observed outputs are indeed spanned by the set of parameter infinitesimals $\boldsymbol{\chi}=(\chi_{1},\ldots,\chi_{p})$ defining all parameter symmetries of the output system. Because $\rho$ may be chosen arbitrarily small, there exists a neighbourhood $\mathcal{B}_{\rho}(\boldsymbol{\theta})$ that is entirely covered by parameter symmetries, which implies local structural identifiability of $\theta_{\ell}$.
  \end{proof}

We say that a model is locally structurally identifiable if all parameters are locally structurally identifiable. In light of Theo.~\ref{thm:structural identifiability_symmetries}, a model is locally structurally identifiable if all parameters are universal parameter invariants. Given our previous discussion about the number of parameter invariants of parameter symmetries, an equivalent formulation is that \textit{a model is locally structurally identifiable if the only parameter symmetry of its input-output system is the trivial parameter symmetry $\Gamma_{\varepsilon}^{\boldsymbol{\theta},0}:\boldsymbol{\theta}\mapsto\boldsymbol{\theta}$}. 

When a model is locally structurally non-identifiable, it is of interest to find the structurally identifiable parameter groupings or parameter quantities. Assume that the parameter symmetries $\Gamma_{\varepsilon}^{\boldsymbol{\theta}}$ of interest have $\tilde{p}\leq{p}$ universal parameter invariants denoted by $I_{k}$, $k\in\{1,\ldots,\tilde{p}\}$ which are collected in a vector $\vec{I}_{\boldsymbol{\theta}}$. Crucially, we can always re-parametrise outputs $\mathbf{y}(t,\boldsymbol{\theta})$ in terms of these universal parameter invariants giving us $\mathbf{y}(t,\vec{I}_{\boldsymbol{\theta}})$, and then apply Theo.~\ref{thm:structural identifiability_symmetries} on the re-parametrised outputs to give:

\medskip
\begin{cor}[Local structural identifiability of parameter combinations in terms of universal parameter invariants]
The \textit{locally structurally identifiable} parameter quantities of the output ODE system in Eq.~\eqref{eq:ODE_sys_output} are given by its universal parameter invariants.
\label{cor:structural_identifiability_symmetries}
\end{cor}

\medskip
The result in Cor. \ref{cor:structural_identifiability_symmetries} highlights an important difference between the standard differential algebra approach for determining structural identifiability and the approach based on parameter symmetries. The former approach finds the globally identifiable parameter quantities whereas the latter finds locally structurally identifiable parameter quantities. Moreover, the globally structurally identifiable parameter quantities can always be written as a function of the locally structurally identifiable parameter quantities and, based on Cor. \ref{cor:structural_identifiability_symmetries}, this means that any globally structurally identifiable parameter quantity can indeed be written as a function of the universal parameter invariants. Subsequently, we show that these conclusions are also true in the general case. We begin by providing an explanation for why the standard differential algebra approach always finds universal parameter invariants.

\subsection{The standard differential algebra approach finds universal parameter invariants}

We demonstrate that the the standard differential algebra approach for analysing structural identifiability will always find universal parameter invariants (Theo. \ref{thm:structural identifiability_symmetries} and Cor. \ref{cor:structural_identifiability_symmetries}). To this end, consider a system of output ODEs in Eq.~\eqref{eq:ODE_sys_output} where the function $\Delta$ is a vector-valued multivariate polynomial for which the monomials are composed of the outputs $\mathbf{y}$ and derivatives of the outputs. In this case, the standard differential algebra approach is conducted in two steps where we first collect all coefficients of the monomials, and then simplify these coefficients algebraically. This procedure yields all identifiable parameter quantities, and here we show that these identifiable parameter quantities are always given by universal parameter invariants according to our definition of local structural identifiability (Cor.~\ref{cor:structural_identifiability_symmetries}). This fact is a consequence of two fundamental properties of differential invariants of Lie symmetries.

One of the most fundamental properties of invariants is that \textit{any function of differential invariants is itself a differential invariant}, and this is well-known within the field of classical symmetries~\cite{bluman1989symmetries}. Of course, the same property also holds for parameter symmetries, and here we present this result as a proposition for the sake of completeness.  

\medskip
\begin{prop}[Functions of differential invariants are themselves differential invariants]
Let $X_{\boldsymbol{\theta}}$ generate a parameter symmetry with $p$ parameters and denote its functionally independent parameter invariants by $I_{1},\ldots,I_{\tilde{p}}$ where the number of invariants is an integer $\tilde{p}\in\{2,\ldots,p\}$. Then any differentiable function $F\left(I_{1},\ldots,I_{\tilde{p}}\right)$ of these invariants is itself a differential invariant.
\label{prop:inv}
\end{prop}
\begin{proof}
By definition, we have that $X_{\boldsymbol{\theta}}(I_{j})=0$ $\forall j\in\left\{1,\ldots,\tilde{p}\right\}$ and we need to show that $X_{\boldsymbol{\theta}}\left(F\left(I_{1},\ldots,I_{\tilde{p}}\right)\right)=0$. By the chain rule, it follows that
\begin{equation}
\frac{\partial F}{\partial\theta_{\ell}}=\frac{\partial}{\partial\theta_{\ell}}\left(F\left(I_{1},\ldots,I_{\tilde{p}}\right)\right)=\sum_{j=1}^{\tilde{p}}\frac{\partial I_{j}}{\partial\theta_{\ell}}\frac{\partial F}{\partial I_{j}},\quad\ell\in\left\{1,\ldots,p\right\}.
\end{equation}
Thus, we have
\begin{align}
X_{\boldsymbol{\theta}}\left(F\left(I_{1},\ldots,I_{\tilde{p}}\right)\right)&=\sum_{\ell=1}^{p}\chi_{\ell}\frac{\partial F}{\partial\theta_{\ell}}\nonumber\\
&=\sum_{\ell=1}^{p}\chi_{\ell}\sum_{j=1}^{\tilde{p}}\frac{\partial I_{j}}{\partial\theta_{\ell}}\frac{\partial F}{\partial I_{j}}\nonumber\\
&=\sum_{\ell=1}^{p}\sum_{j=1}^{\tilde{p}}\chi_{\ell}\frac{\partial I_{j}}{\partial\theta_{\ell}}\frac{\partial F}{\partial I_{j}}\nonumber\\
    &=\sum_{j=1}^{\tilde{p}}\frac{\partial F}{\partial I_{j}}\underset{=X_{\boldsymbol{\theta}}(I_{j})}{\underbrace{\left(\sum_{\ell=1}^{p}\chi_{\ell}\frac{\partial I_{j}}{\partial\theta_{\ell}}\right)}}\nonumber\\
    &=\sum_{j=1}^{\tilde{p}}\frac{\partial F}{\partial I_{j}}\underset{=0}{\underbrace{X_{\boldsymbol{\theta}}(I_{j})}}=0,
\end{align}
which is the desired result.
\end{proof}

As a short detour, we clarify the difference between parameter invariants (Defn.~\ref{def:param_inv}) and universal parameter invariants (Defn.~\ref{def:universal}) using this fundamental property of differential invariants. Previously, we saw that the trivial parameter symmetry $\Gamma_{\varepsilon}^{\boldsymbol{\theta},0}:\boldsymbol{\theta}\mapsto\boldsymbol{\theta}$ is always a parameter symmetry of all possible models, and, more importantly, all parameters $\boldsymbol{\theta}$ are parameter invariants of $\Gamma_{\varepsilon}^{\boldsymbol{\theta},0}$. Now, take another parameter symmetry $\Gamma_{\varepsilon}^{\boldsymbol{\theta}}$ that has a parameter invariant $I=F(\boldsymbol{\theta})$ defined by some non-constant, non-linear, multivariate and arbitrary function $G$. Moreover, let us assume that $I=G(\boldsymbol{\theta})$ is a universal parameter invariant. Then, the parameter invariant $I=G(\boldsymbol{\theta})$ of $\Gamma_{\varepsilon}^{\boldsymbol{\theta}}$ is also a parameter invariant of the trivial parameter symmetry $\Gamma_{\varepsilon}^{\boldsymbol{\theta},0}$ since $I$ is a function of the parameter invariants $\boldsymbol{\theta}$ of $\Gamma_{\varepsilon}^{\boldsymbol{\theta},0}$ (Prop.~\ref{prop:inv}). For the same reason, all parameter invariants of $\Gamma_{\varepsilon}^{\boldsymbol{\theta}}$ are necessarily parameter invariants of $\Gamma_{\varepsilon}^{\boldsymbol{\theta},0}$. Nevertheless, the converse statement is not true as many of the parameter invariants of the trivial symmetry $\Gamma_{\varepsilon}^{\boldsymbol{\theta},0}$ are not shared with $\Gamma_{\varepsilon}^{\boldsymbol{\theta}}$. 

Using the fact that any function of invariants is itself an invariant, we draw important conclusions about the structure of the system of output ODEs by analysing the linearised symmetry conditions defining parameter symmetries. By definition, the function $\Delta$ in Eq.~\eqref{eq:ODE_sys_output} defining the output-system of ODEs solves the linearised symmetry conditions in Eq.~\eqref{eq:lin_sym_parameter} defining all parameter symmetries. This implies that the function $\Delta$ is, in fact, a universal invariant. As a consequence of Prop. \ref{prop:inv}, the system of output ODEs can therefore always be written as a function of universal invariants.

\medskip
\begin{prop}[ODEs as functions of universal invariants]
Consider the system of output ODEs defined by a function $\Delta$ as in Eq.~\eqref{eq:ODE_sys_output}. Assume that this system of output ODEs has $\tilde{p}\in\{1,\ldots,p\}$ functionally independent universal parameter invariants, and denote these by $I_{k},k\in\{1,\ldots,\tilde{p}\}$ which are collected in a vector $\vec{I}_{\boldsymbol{\theta}}$. Then, $\Delta$ is a function $\Phi$ of the universal invariants according to
\begin{equation}
\Delta\left(t,\frac{\mathrm{d}^{N}\mathbf{y}}{\mathrm{d}t^{N}},\frac{\mathrm{d}^{N-1}\mathbf{y}}{\mathrm{d}t^{N-1}},\ldots,\frac{\mathrm{d}\mathbf{y}}{\mathrm{d}t},\mathbf{y},\boldsymbol{\theta}\right)=\Phi\left(t,\frac{\mathrm{d}^{N}\mathbf{y}}{\mathrm{d}t^{N}},\frac{\mathrm{d}^{N-1}\mathbf{y}}{\mathrm{d}t^{N-1}},\ldots,\frac{\mathrm{d}\mathbf{y}}{\mathrm{d}t},\mathbf{y},\vec{I}_{\boldsymbol{\theta}}\right).
\label{eq:invariant_ODE}
\end{equation}
\label{prop:ODE}
\end{prop}

Armed with Prop.~\ref{prop:ODE}, we understand why the standard differential algebra approach for conducting a global structural identifiability analysis finds identifiable parameter quantities from a symmetry perspective. Again, consider the case discussed previously when the system of output ODEs in Eq.~\eqref{eq:ODE_sys_output} is defined by a function $\Delta$ that is a vector-valued multivariate polynomial. The standard differential algebra approach for elucidating structural identifiability extracts coefficients of the monomials, and by virtue of Eq.~\eqref{eq:invariant_ODE} these coefficients must either be a constant or a universal parameter invariant. Better still, our notion of structural identifiability in terms of universal parameter invariants (Theo.~\ref{thm:structural identifiability_symmetries} and Cor.~\ref{cor:structural_identifiability_symmetries}) does not only yield locally identifiable parameter quantities, but it also allows us to characterise the parameter transformations preserving the observed outputs in the form of parameter symmetries. 

Thus far, we have introduced a theoretical framework for analysing local structural identifiability through parameter symmetries. We now turn to concrete examples that demonstrate how this parameter-symmetry-based approach can be implemented in practice to assess local structural identifiability.

\section{Analysing local structural identifiability using parameter symmetries}

Subsequently, we use the toy example presented in Section 1, see Eqs.~\eqref{eq:toy_individual} and~\eqref{eq:toy_intro}, to illustrate that calculating universal parameter invariants provides information that is identical with the findings of the standard differential algebra approach when the globally and locally structurally identifiable parameter combinations are the same. Thereafter, we demonstrate through the use of a second example that the conclusions of the approach based on calculating universal parameter invariants and the standard differential algebra approach are different when the globally and locally structurally identifiable parameter combinations are different. We proceed by generalising the approach for analysing local structural identifiability using parameter symmetries implemented in these initial two examples, by presenting a three-step procedure called the CaLinInv recipe. Using this three-step procedure, we analyse the local structural identifiability of two larger models.

\subsection{Analysing the structural identifiability of a toy model using parameter symmetries}

Consider the toy example in Section~1 that results in the output system in Eq.~\eqref{eq:toy_intro} given by $\mathrm{d}y/\mathrm{d}t=(\kappa_{1}+\kappa_{2})-\lambda{y}$. The standard differential algebra approach concluded that the parameter $\lambda$ is globally structurally identifiable and that the parameters $\kappa_{1}$ and $\kappa_{2}$ are individually non-identifiable while their sum $\kappa_{1}+\kappa_{2}$ is globally structurally identifiable.

As a comparison, we next consider the symmetry-based approach for elucidating structural identifiability which entails finding universal parameter invariants of the output ODE in Eq.~\eqref{eq:toy_intro}. To this end, we look for parameter symmetries $\Gamma_{\varepsilon}^{\boldsymbol{\theta}}$ of the output ODE $\mathrm{d}y/\mathrm{d}t=(\kappa_{1}+\kappa_{2})-\lambda{y}$ with the following structure
\begin{equation}
    \Gamma_{\varepsilon}^{\boldsymbol{\theta}}:(\kappa_{1},\kappa_{2},\lambda)\mapsto(\hat{\kappa}_{1}(\kappa_{1},\kappa_{2},\lambda,\varepsilon),\hat{\kappa}_{2}(\kappa_{1},\kappa_{2},\lambda,\varepsilon),\hat{\lambda}(\kappa_{1},\kappa_{2},\lambda,\varepsilon)).
    \label{eq:parameter_symmetry}
\end{equation}
We denote the generating vector field of the parameter symmetry $\Gamma_{\varepsilon}^{\boldsymbol{\theta}}$ in Eq.~\eqref{eq:parameter_symmetry} by
\begin{equation}
    X_{\boldsymbol{\theta}}=\chi_{\kappa_{1}}(\kappa_{1},\kappa_{2},\lambda)\partial_{\kappa_{1}}+\chi_{\kappa_{2}}(\kappa_{1},\kappa_{2},\lambda)\partial_{\kappa_{2}}+\chi_{\lambda}(\kappa_{1},\kappa_{2},\lambda)\partial_{\lambda}.
    \label{eq:X_theta}
\end{equation}
Given this vector field, we consider the following linearised symmetry condition:
\begin{equation}
    X_{\boldsymbol{\theta}}\left[  \dfrac{\mathrm{d}y}{\mathrm{d}t}-((\kappa_{1}+\kappa_{2})-\lambda{y})\right]=0\quad\textrm{whenever}\quad  \dfrac{\mathrm{d}y}{\mathrm{d}t}-((\kappa_{1}+\kappa_{2})-\lambda{y})=0,
    \label{eq:lin_sym_toy_1}
\end{equation}
which states that the solution manifold is invariant under transformations by the parameter symmetry $\Gamma_{\varepsilon}^{\boldsymbol{\theta}}$ in Eq.~\eqref{eq:parameter_symmetry}. Carrying out the differentiation on the left-hand side yields the following equivalent equation:
\begin{equation}(\chi_{\kappa_{1}}+\chi_{\kappa_{2}})+y\chi_{\lambda}=0.
    \label{eq:lin_sym_toy_2}
\end{equation}
Moreover, since the monomials $\{1,y\}$ are linearly independent the above equation implies that the following two equations must hold simultaneously,
\begin{equation}
    \chi_{\kappa_{1}}=-\chi_{\kappa_{2}},\quad\chi_{\lambda}=0,
    \label{eq:lin_sym_toy_3}
\end{equation}
and thus the family of generating vector fields is given by
\begin{equation}
    X_{\boldsymbol{\theta}}=\chi_{\kappa_{1}}(\kappa_{1},\kappa_{2},\lambda)(\partial_{\kappa_{1}}-\partial_{\kappa_{2}}),
    \label{eq:X_theta_2}
\end{equation}
for some arbitrary function $\chi_{\kappa_{1}}$ of the parameters. Next, we look for parameter invariants $I(\kappa_{1},\kappa_{2},\lambda)$ satisfying 
\begin{equation}
    X_{\boldsymbol{\theta}}(I(\kappa_{1},\kappa_{2},\lambda))=\chi_{\kappa_{1}}(\kappa_{1},\kappa_{2},\lambda)\left(\frac{\partial I}{\partial\kappa_{1}}-\frac{\partial I}{\partial\kappa_{2}}\right)=0.
    \label{eq:inv_para}
\end{equation}
To find differential invariants, we apply the method of characteristics to Eq.~\eqref{eq:inv_para}. Specifically, we look for a parametrised solution curve $I(s)=I(\kappa_{1}(s),\kappa_{2}(s),\lambda(s))$ where $s$ is an arbitrary parameter. By the chain rule, it follows that
\begin{equation}
    \frac{\mathrm{d}I}{\mathrm{d}s}=\frac{\mathrm{d}\kappa_{1}}{\mathrm{d}s}\frac{\partial I}{\partial\kappa_{1}}+\frac{\mathrm{d}\kappa_{2}}{\mathrm{d}s}\frac{\partial I}{\partial\kappa_{2}}+\frac{\mathrm{d}\lambda}{\mathrm{d}s}\frac{\partial I}{\partial \lambda}.
    \label{eq:chain}    
\end{equation}
By comparing Eqs.~\eqref{eq:inv_para} and~\eqref{eq:chain}, we obtain the following characteristic equations
\begin{align}
    \frac{\mathrm{d}I}{\mathrm{d}s}&=0,\label{eq:chara_1}\\
    \frac{\mathrm{d}\lambda}{\mathrm{d}s}&=0,\label{eq:chara_2}\\
    \frac{\mathrm{d}\kappa_{1}}{\mathrm{d}s}&=\chi_{\kappa_{1}},\label{eq:chara_3}\\
    \frac{\mathrm{d}\kappa_{2}}{\mathrm{d}s}&=-\chi_{\kappa_{1}}.\label{eq:chara_4}
\end{align}
By Eq.~\eqref{eq:chara_1}, any differential invariant is an arbitrary integration constant or a first integral, i.e. $I=\mathrm{constant}$. By Eq.~\eqref{eq:chara_2} it follows that the first differential invariant is given by $I_{1}=\lambda$. By combining Eqs.~\eqref{eq:chara_3} and~\eqref{eq:chara_4} under the assumption that $\chi_{\kappa_{1}}\neq 0$, we obtain 
\begin{equation}
    \frac{\mathrm{d}\kappa_{1}}{\mathrm{d}\kappa_{2}}=-1,
    \label{eq:temp_ODE}
\end{equation}
which is readily integrated to
\begin{equation}
    I_{2}=\kappa_{1}+\kappa_{2},
    \label{eq:temp_solution}
\end{equation}
where $I_{2}$ is an arbitrary integration constant. Since $I_{2}$ is a first integral of Eq.~\eqref{eq:temp_ODE}, this implies that it is also another differential invariant. In total, this implies that the two universal parameter invariants of the parameter symmetry $\Gamma_{\varepsilon}^{\boldsymbol{\theta}}$ in Eq.~\eqref{eq:parameter_symmetry} are given by
\begin{equation}
    I_{1}=\lambda,\quad I_{2}=\kappa_{1}+\kappa_{2},
\label{eq:differential_invariants_toy}
\end{equation}
which agrees with the conclusions of the standard differential algebra approach. Better still, we clearly see how the parameter symmetries in Eq.~\eqref{eq:parameter_symmetry} act on the parameters of the model. For instance, the parameter symmetry defined by $\chi_{\kappa_{1}}=1$ corresponds to translations with opposite signs of the parameters $\kappa_{1}$ and $\kappa_{2}$, respectively, according to
\begin{equation}
    \Gamma_{\varepsilon}^{\boldsymbol{\theta}}:(\kappa_{1},\kappa_{2},\lambda)\mapsto(\kappa_{1}+\varepsilon,\kappa_{2}-\varepsilon,\lambda),
    \label{eq:parameter_symmetry_toy}
\end{equation}
and clearly this symmetry preserves the universal parameter invariants in Eq.~\eqref{eq:differential_invariants_toy} since
\begin{equation}
\hat{\kappa}_{1}(\kappa_{1},\kappa_{2},\lambda,\varepsilon)+\hat{\kappa}_{2}(\kappa_{1},\kappa_{2},\lambda,\varepsilon)=\kappa_{1}+\kappa_{2}=I_{2}\quad\forall\varepsilon\in\mathbb{R}.
  \label{eq:toy_preserve_invariants}
\end{equation}
This toy example illustrates that the standard differential algebra approach and the approach based on parameter symmetries arrive at the same conclusions regarding the identifiable parameter quantities \textit{in the case where the locally and globally structurally identifiable parameter quantities are the same}. Additionally, the symmetry-based approach also yields the family of parameter symmetries corresponding to the parameter transformations that preserve the observed outputs. 

Next, we analyse the structural identifiability of a linear model, for which the locally and globally identifiable parameters are different, by means of the standard differential algebra approach and the approach based on parameter symmetries.


\subsection{Analysing structural identifiability of a linear model where the locally and globally identifiable parameters are different}

{
Let $u(t)$ and $v(t)$ denote the concentrations of two chemical species depending on time $t$ which satisfy the coupled system of two linear ODEs
\begin{equation}
  \begin{split}
    \dfrac{\mathrm{d}u}{\mathrm{d}t}&=au+bv\,,\\
    \dfrac{\mathrm{d}v}{\mathrm{d}t}&=cv\,,\\
  \end{split}
   \label{eq:toy}
 \end{equation}
where $a,b,c$ are three positive rate constants. Furthermore, assume that we observe the first species, i.e.~$y(t)=u(t)$, yielding the following model for the output
 \begin{equation}
   0=\dfrac{\mathrm{d}^{2}y}{\mathrm{d}t^{2}}-(a+c)\dfrac{\mathrm{d}y}{\mathrm{d}t}+(ac)y\,.
   \label{eq:toy_output}
 \end{equation}
 We first apply the standard differential algebra approach for analysing the structural identifiability of this model. First, we see that the parameter $b$ does not appear in the output equation in Eq. \eqref{eq:toy_output}, and hence \textit{$b$ is structurally non-identifiable. Next, we extract the coefficients in front of $\{\mathrm{d}^{2}y/\mathrm{d}t^{2},\mathrm{d}y/\mathrm{d}t,y\}$ resulting in the set $\{1,\mathcal{A},\mathcal{B}\}$ where $\mathcal{A}=a+c$ is the sum and $\mathcal{B}=ac$ is the product of the parameters $a$ and $c$. Clearly, the globally structurally identifiable parameter quantities correspond to the sum $\mathcal{A}$ and the product $\mathcal{B}$, whereas the individual parameters $a$ and $c$ are not globally structurally identifiable. This can be seen by equating $a=\mathcal{A}-c$ with $a=\mathcal{B}/c$ which yields a quadratic equation for the parameter $c$:}
\begin{equation}
  c^{2}-c\mathcal{A}+\mathcal{B}=0\,.
  \label{eq:quadratic_c}
  \end{equation}
  The two solutions of this equation are given by
  \begin{equation}
    c=\frac{1}{2}\left(\mathcal{A}\pm\sqrt{\mathcal{A}^{2}-4\mathcal{B}}\right)\,,
    \label{eq:quadratic_solution}
  \end{equation}  
   and thus the two sets of parameter values that yield the same outputs are given by
   \begin{equation} 
   (a,c)=\left(\dfrac{2\mathcal{B}}{\left(\mathcal{A}\pm\sqrt{\mathcal{A}^{2}-4\mathcal{B}}\right)},\frac{1}{2}\left(\mathcal{A}\pm\sqrt{\mathcal{A}^{2}-4\mathcal{B}}\right)\right)\,.
     \label{eq:global_parameters}
   \end{equation}
   Thus, the parameters $a$ and $c$ are only locally structurally identifiable. Furthermore, since parameter symmetries act on parameters locally in an infinitesimal fashion, it is not surprising that their universal parameter invariants indeed correspond to the locally identifiable parameter quantities. This is clearly illustrated using the symmetry-based method.

We now analyse the structural identifiability of the output system in Eq.~\eqref{eq:toy_output} using parameter symmetries. The linearised symmetry condition is given by
\begin{equation}
 0=-(\chi_{a}+\chi_{c})\dfrac{\mathrm{d}y}{\mathrm{d}t}+(c\chi_{a}+a\chi_{c})y\,.
 \label{eq:lin_sym_toy}
\end{equation}
In order for a parameter to be locally structurally identifiable, we require its infinitesimal to be identically equal to zero. Since the infinitesimal $\chi_{b}$ does not appear in the above equation, it follows that $b$ is non-identifiable. Decomposing Eq.~\eqref{eq:lin_sym_toy} with respect to the linearly independent set $\{\mathrm{d}y/\mathrm{d}t,y\}$ yields
\begin{equation}
 \chi_{a}=-\chi_{c}\,,\quad\chi_{a}=-\left(\dfrac{a}{c}\right)\chi_{c}\,.
 \label{eq:det_eq_toy}
\end{equation}
Equating these two equations yields
$$\left(1-\dfrac{a}{c}\right)\chi_{c}=0\,,$$
and when $a\neq{c}$ it follows that $\chi_{c}=0$ and thus $c$ is locally structurally identifiable. Since $\chi_{a}=-\chi_{c}$, we have that $\chi_{a}=0$ and thus $a$ is also locally structurally identifiable. Thus, when $a\neq{c}$ both $a$ and $c$ are locally structurally identifiable.

When $a=c$ we see that the output ODE in Eq. \eqref{eq:toy_output} becomes $0=\mathrm{d}^{2}y/\mathrm{d}t^{2}-2a(\mathrm{d}y/\mathrm{d}t)+a^{2}y$  and the corresponding linearised symmetry condition becomes $0=-2\chi_{a}(\mathrm{d}y/\mathrm{d}t)+2a\chi_{a}y$. By separating the last linearised symmetry condition with respect to $\{\mathrm{d}y/\mathrm{d}t,y\}$, it decomposes into the single constraint $\chi_{a}=0$ implying that $a$ is locally structurally identifiable. Since $a=c$, $c$ must also be locally structurally identifiable.

In conclusion, this example highlights the key difference between the standard differential algebra approach to structural identifiability and the parameter-symmetry-based approach. The former identifies globally identifiable parameter combinations, whereas the latter reveals locally structurally identifiable parameter quantities. Next, we generalise the parameter-symmetry-based approach for analysing local structural identifiability.

\subsection{A parameter-symmetry-based approach for elucidating local structural identifiability}

We present a recipe in three steps for elucidating local structural identifiability using parameter symmetries. These three steps are captured by the acronym \textit{CaLinInv}; \textit{Ca}nonical coordinates, \textit{Lin}earised symmetry conditions and differential \textit{Inv}ariants (Alg.~\ref{alg:CaLinInv}). Importantly, the CaLinInv recipe is by no means restricted to polynomial systems of output ODEs and thus it works on arbitrary output systems. We want to emphasise that the additional information that is gained when using the CaLinInv recipe over the differential algebra approach is the parameter symmetries or, differently put, the parameter transformations which preserve the observed outputs. In the particular case when $\Delta$ defining the system of ODEs for the outputs in Eq.~\eqref{eq:ODE_sys_output} is composed of multivariate polynomials, the linearised symmetry conditions decompose into a system of linear equations that can be solved using Gaussian elimination. 

\begin{algorithm}[ht!]
\caption{The \textit{CaLinInv}-method for a symmetry-based local structural identifiability-analysis.}
\smallskip
\KwIn{A system of first-order ODEs with associated observed outputs as in Eq.~\eqref{eq:ODE_sys_full}.}
\KwOut{The locally identifiable parameter quantities corresponding to universal parameter invariants and the transformations preserving these parameter quantities in the form of a family of parameter symmetries $\Gamma_{\varepsilon}^{\boldsymbol{\theta}}$.}
\textit{Step 1: \textbf{Ca}nonical coordinates}. The outputs $\mathbf{y}$ are canonical coordinates of parameter symmetries. Re-write the original system of first-order ODEs in Eq.~\eqref{eq:ODE_sys_full} as a system of ODEs as in Eq.~\eqref{eq:ODE_sys_output} solely depending on the observed outputs $\mathbf{y}$ and the parameters $\boldsymbol{\theta}$.\\
\textit{Step 2: \textbf{Lin}earised symmetry conditions}. Find the parameter symmetries $\Gamma_{\varepsilon}^{\boldsymbol{\theta}}$ of the resulting system of output ODEs which are generated by $X_{\boldsymbol{\theta}}=\sum_{\ell=1}^{p}\chi_{\ell}(\boldsymbol{\theta})\partial_{\theta_{\ell}}$. In other words, solve the linearised symmetry conditions for the infinitesimals $\chi_{\ell}$ for $\ell\in\left\{1,\ldots,p\right\}$ and then use $X_{\boldsymbol{\theta}}$ to generate $\Gamma_{\varepsilon}^{\boldsymbol{\theta}}$.\\
\textit{Step 3: Universal \textbf{Inv}ariants}. Find the universal parameter invariants $I=I\left(\boldsymbol{\theta}\right)$ of the parameter symmetries by solving the linear PDE $X_{\boldsymbol{\theta}}(I)=0$ using the method of characteristics.
\medskip
\label{alg:CaLinInv}
\end{algorithm}

In the case where $\Delta$ in Eq.~\eqref{eq:ODE_sys_output} consists of rational functions of the outputs $\mathbf{y}$ and their derivatives, the linearised symmetry conditions 
\begin{equation}
X_{\boldsymbol{\theta}}\left(\Delta\left(t,\frac{\mathrm{d}^{N}\mathbf{y}}{\mathrm{d}t^{N}},\frac{\mathrm{d}^{N-1}\mathbf{y}}{\mathrm{d}t^{N-1}},\ldots,\frac{\mathrm{d}\mathbf{y}}{\mathrm{d}t},\mathbf{y},\boldsymbol{\theta}\right)\right)=\mathbf{0},
        \label{eq:lin_sym_outputs}
\end{equation}
decompose into a linear system of equations of the form
\begin{equation}
        M\boldsymbol{\chi}=\mathbf{0},
\label{eq:mat_system_outputs}
\end{equation}
where $M$ is a $\tilde{n}\times p$-matrix, $\boldsymbol{\chi}$ is a $p\times 1$-vector and $\mathbf{0}$ is the $\tilde{n}\times 1$-zero vector. Here, $p$ is the number of parameters and $\boldsymbol{\chi}=\left(\chi_{1},\chi_{2},\ldots,\chi_{p}\right)$ contains the parameter infinitesimals. 
    \label{remark:matrix}

The first step in the CaLinInv recipe, expressing the original system as an equivalent system solely depending on the observed outputs, is identical to the first step in the standard differential algebra approach. Thereafter, the methodologies differ where the differential algebra approach simply extracts the coefficients of the monomials and then reduces the resulting set of parameter combinations, whereas the CaLinInv recipe solves the linearised symmetry conditions, generates parameter symmetries and calculates universal parameter invariants.  In terms of outcomes, both approaches yield the identifiable parameter quantities corresponding to universal parameter invariants. However, the CaLinInv recipe yields the family of parameter symmetries whereas the standard differential algebra approach merely yields the universal parameter invariants. 


\subsection{Analysing local structural identifiability of a glucose-insulin model with a time-dependent input}

Next, we study a model of glucose-insulin regulation that was originally presented in~\cite{bolie1961coefficients}. Importantly, this model has been subject to structural identifiability analyses using the differential algebra approach~\cite{cobelli1980parameter} as well as a symmetry-based analysis focusing on full symmetries~\cite{massonis2020finding}. Here, we study this model by means of parameter symmetries instead and we characterise the family of parameter symmetries preserving the observed outputs. 


\subsubsection*{Step 1: Canonical coordinates}

In this model, we have two states given by $x_{1}(t)$, glucose concentration, and $x_{2}(t)$, insulin concentration, one known input $u(t)\geq{0}$ corresponding to glucose entering the digestive system, and one output $y(t)$ corresponding to a glucose measurement. In total, there are five parameters
\begin{equation}
\boldsymbol{\theta}=(p_{1},p_{2},p_{3},p_{4},V_{p}),
\label{eq:theta_glucose}
\end{equation}
where the first four encode first-order reaction rates while the last parameter corresponds to the volume of blood extracted during glucose measurement, and the original system of ODEs is given by
\begin{equation}
\begin{split}
    \dfrac{\mathrm{d}x_{1}}{\mathrm{d}t}&=u+p_{1}x_{1}-p_{2}x_{2},\\
    \dfrac{\mathrm{d}x_{2}}{\mathrm{d}t}&=p_{3}x_{2}+p_{4}x_{1},\\
    y&=\frac{x_{1}}{V_{p}}.
\end{split}
\end{equation}
Note that since the input $u(t)$ depends on time the model of interest is non-autonomous, and we assume that the input is nonconstant. Furthermore, we assume that the input and the output are linearly independent, i.e. $u(t)\neq{C}y(t)$ for some $t\in\mathbb{R}$ and some constant $C\in\mathbb{R}$. The ODE for the output is given by
\begin{equation}
V_{p} \dfrac{\mathrm{d}^{2}y}{\mathrm{d}t^{2}} - V_{p}(p_{1} +p_{3}) \dfrac{\mathrm{d}y}{\mathrm{d}t} + V_{p}(p_{1}p_{3}+p_{2}p_{4})y+p_{3}u-\dfrac{\mathrm{d}u}{\mathrm{d}t}=0\,.
    \label{eq:glucose_ODE_output}
\end{equation}
\subsubsection*{Step 2: Linearised symmetry condition}
The linearised symmetry condition is given by
\begin{equation}
\begin{split}
\left[-V_{p}\dfrac{\mathrm{d}y}{\mathrm{d}t}+V_{p}p_{3}y\right]\chi_{p_{1}}+[V_{p}p_{4}y]\chi_{p_{2}}+\left[-V_{p}\dfrac{\mathrm{d}y}{\mathrm{d}t}+V_{p}p_{1}y\right]\chi_{p_{3}}+u\chi_{p_{3}}+\left[V_{p}p_{2}y\right]\chi_{p_{4}}&\\
+\left[\dfrac{\mathrm{d}^{2}y}{\mathrm{d}t^{2}}-(p_{1}+p_{3})\dfrac{\mathrm{d}y}{\mathrm{d}t}+(p_{1}p_{3}+p_{2}p_{4})y\right]\chi_{V_{p}}&=0.\\
\end{split}
    \label{eq:glucose_lin_sym}
\end{equation}
Therefore, the linearly independent set of coefficients is those relating to $\{\mathrm{d}^{2}y/\mathrm{d}t^{2},\mathrm{d}y/\mathrm{d}t,y,u\}$. The coefficient of $\mathrm{d}^{2}y/\mathrm{d}t^{2}$ shows that $\chi_{V_{p}}=0$ and therefore $V_{p}$ is locally structurally identifiable. Moreover, the coefficient of the input $u$ yields that $\chi_{p_{3}}=0$ and therefore $p_{3}$ is also structurally identifiable. By substituting $\chi_{V_{p}}=\chi_{p_{3}}=0$ into Eq.~\eqref{eq:glucose_lin_sym}, we obtain 
\begin{equation}
    \dfrac{\mathrm{d}y}{\mathrm{d}t}\left(-V_{p}\chi_{p_{1}}\right)+y\left(V_{p}p_{3}\chi_{p_{1}}+V_{p}p_{4}\chi_{p_{2}}+V_{p}p_{2}\chi_{p_{4}}\right)=0.
    \label{eq:glucose_lin_sym_2}
\end{equation}
The coefficient of $\mathrm{d}y/\mathrm{d}t$ yields that $\chi_{p_{1}}=0$ and hence $p_{1}$ is locally structurally identifiable. Lastly, the coefficient of $y$ together with the conclusion that $\chi_{p_{1}}=0$ yields 
\begin{equation}
    \chi_{p_{4}}=-\left(\frac{p_{4}}{p_{2}}\right)\chi_{p_{2}},
    \label{eq:glucose_lin_sym_3}
\end{equation}
and hence the generating vector fields of the family of symmetries of the glucose-insulin model are given by
\begin{equation}
    X_{\boldsymbol{\theta}}=\frac{1}{p_{2}}\chi_{p_{2}}\left(p_{1},p_{2},p_{3},p_{4},V_{p}\right)\left[p_{2}\partial_{p_{2}}-p_{4}\partial_{p_{4}}\right],
    \label{eq:X_glucose}
\end{equation}
for arbitrary functions $\chi_{p_{2}}$ of the parameters $\boldsymbol{\theta}$ in Eq.~\eqref{eq:theta_glucose}. These parameter symmetries correspond to scalings of the parameters $p_{2}$ and $p_{4}$. To illustrate this, consider the parameter symmetry $\Gamma_{\varepsilon}^{\boldsymbol{\theta}}$ defined by the arbitrary function $\chi_{p_{2}}=p_{2}$. Substituting $\chi_{p_{2}}=p_{2}$ into the vector field $X_{\boldsymbol{\theta}}$ in Eq.~\eqref{eq:X_glucose} results in $X_{\boldsymbol{\theta}}=p_{2}\partial_{p_{2}}-p_{4}\partial_{p_{4}}$, and generating the corresponding symmetry yields
\begin{equation}\Gamma_{\varepsilon}^{\boldsymbol{\theta}}:(p_{1},p_{2},p_{3},p_{4},V_{p})\mapsto(p_{1},p_{2}\exp(\varepsilon),p_{3},p_{4}\exp(-\varepsilon),V_{p}).
    \label{eq:Gamma_glucose}
\end{equation}
}


\subsubsection*{Step 3: Universal Invariants}

Thus far, we have calculated three universal parameter invariants corresponding to the directly identifiable parameters: $I_{1}=p_{1}$, $I_{2}=p_{3}$ and $I_{3}=V_{p}$. Next, we find the last universal parameter invariant $I_{4}=I_{4}\left(p_{1},p_{2},p_{3},p_{4},V_{p}\right)$ by solving the equation $X_{\boldsymbol{\theta}}(I_{4})=0$. The method of characteristics yields the following characteristic equation for the remaining differential invariant
\begin{equation}
    \frac{\mathrm{d}p_{2}}{\mathrm{d}p_{4}}=-\frac{p_{2}}{p_{4}},
    \label{eq:cara_glucose}
\end{equation}
and thus the final differential invariant which is a first integral of Eq.~\eqref{eq:cara_glucose} is given by
\begin{equation}
    I_{4}=p_{2}p_{4}.
    \label{eq:invariant_glucose}
\end{equation}
Notably, the parameter symmetry $\Gamma_{\varepsilon}^{\boldsymbol{\theta}}$ in Eq.~\eqref{eq:Gamma_glucose} preserves this last invariant as
\begin{equation}
    \hat{p}_{2}(\varepsilon)\hat{p}_{4}(\varepsilon)=p_{2}p_{4}=I_{4}\quad\forall\varepsilon\in\mathbb{R}.
\end{equation}
In conclusion, the parameters $p_{1}$, $p_{3}$ and $V_{p}$ are locally structurally identifiable. The parameters $p_{2}$ and $p_{4}$ are structurally non-identifiable whereas their product $p_{2}p_{4}$ is locally structurally identifiable.

The glucose-insulin model considered here consists of two first-order ODEs, one output equation and five parameters, and for such a small model we can calculate the universal parameter invariants and parameter symmetries using simple calculations by hand. For larger models with more equations and parameters, the linearised symmetry conditions decompose into a matrix system which can be solved using Gaussian elimination. We next demonstrate this fact by analysing the local structural identifiability of a more complicated model. 


\subsection{Analysing local structural identifiability of an epidemiological model}

We study an SEI model of the epidemiology of tuberculosis~\cite{renardy2022structural}. This model consists of three states: $S(t)$ corresponds to the susceptible population; $E(t)$ corresponds to the exposed population; and $I(t)$ corresponds to the infected population. Moreover, this model has seven parameters: $c$ is the birth rate; $\beta$ is the transmission rate; $\upsilon$ is the probability of primary infection; $\delta$ is the reactivation rate and $\mu_{S}$, $\mu_{E}$ and $\mu_{I}$ are death rates. All parameters are assumed to be positive.

We implement the CaLinInv recipe (Alg.~\ref{alg:CaLinInv}) to analyse structural identifiability. Importantly, we compare the outcomes of these calculations to those obtained through the standard differential algebra approach. Also, we generate and visualise parameter symmetries of this model. The underlying semi-manual calculations were conducted using the open-source symbolic solver \textit{SymPy}~\cite{meurer2017sympy}, and they are presented in the appendices. Details and relevant scripts are  available at the public GitHub-repository associated with this project: \url{https://github.com/JohannesBorgqvist/symmetries_and_structural_identifiability}. 


\subsubsection*{Step 1: Canonical coordinates}

The original system of ODEs is given by
\begin{equation}
    \begin{split}
\dfrac{\mathrm{d}S}{\mathrm{d}t}&=c-\beta SI-\mu_{S}S,\\
\dfrac{\mathrm{d}E}{\mathrm{d}t}&=(1-\upsilon)\beta SI-\delta E-\mu_{E}E,\\
\dfrac{\mathrm{d}I}{\mathrm{d}t}&=\upsilon\beta SI+\delta E-\mu_{I}I.\\
\end{split}
    \label{eq:SEI}
\end{equation}
We consider two outputs given by a proportion of the exposed population and a proportion of the infected population, where these proportions are encoded by the parameters $k_{E}$ and $k_{I}$, respectively. Then, the two observed outputs denoted by $y_{E}$ and $y_{I}$ are given by
\begin{align}
    y_{E}&=k_{E}E,\label{eq:output_E}\\
    y_{I}&=k_{I}I.\label{eq:output_I}
\end{align}
In total, the system has nine parameters so that
\begin{equation}
\boldsymbol{\theta}=(c,\beta,\delta,\upsilon,\mu_{S},\mu_{E},\mu_{I},k_{E},k_{I}).
\label{eq:theta_SEI}
\end{equation}
The system of output ODEs is given by
\begin{align}
\dfrac{\mathrm{d}y_{E}}{\mathrm{d}t} =&-\frac{\delta y_{E}}{\upsilon} - \frac{k_{E} \mu_{I} y_{I}}{k_{I}} - \frac{k_{E}}{k_{I}}\dfrac{\mathrm{d}y_{I}}{\mathrm{d}t} - \mu_{E} y_{E} + \frac{k_{E} \mu_{I} y_{I}}{\upsilon k_{I}} + \frac{k_{E}}{\upsilon k_{I}}\dfrac{\mathrm{d}y_{I}}{\mathrm{d}t},\label{eq:ODE_y_E_2}\\  
\dfrac{\mathrm{d}^{2}y_{I}}{\mathrm{d}t^{2}} =&\textcolor{white}{-}\beta c \upsilon y_{I} + \frac{\beta \delta y_{E} y_{I}}{k_{E}} - \frac{\beta \mu_{I} y_{I}^{2}}{k_{I}} - \frac{\beta y_{I}}{k_{I}}\dfrac{\mathrm{d}y_{I}}{\mathrm{d}t} + \frac{\delta k_{I} \mu_{S} y_{E}}{k_{E}} - \frac{\delta k_{I} y_{E}}{k_{E} y_{I}}\dfrac{\mathrm{d}y_{I}}{\mathrm{d}t} + \frac{\delta k_{I}}{k_{E}}\dfrac{\mathrm{d}y_{E}}{\mathrm{d}t}\nonumber\\
&-\mu_{I} \mu_{S} y_{I} - \mu_{S} \dfrac{\mathrm{d}y_{I}}{\mathrm{d}t} + \frac{1}{y_{I}}\left(\dfrac{\mathrm{d}y_{I}}{\mathrm{d}t}\right)^{2}.\label{eq:ODE_y_I_2}
\end{align}


\subsubsection*{Step 2: Linearised symmetry conditions}

The two linearised symmetry conditions can be simplified to the following two equations
\begin{align}
0 =&\textcolor{white}{-}\beta c k_{E}^{2} k_{I}^{2} \chi_{\upsilon} y_{I}^{2} - \beta \delta k_{I}^{2} \chi_{k_{E}} y_{E} y_{I}^{2} + \beta \upsilon k_{E}^{2} k_{I}^{2} \chi_{c} y_{I}^{2} - \beta k_{E}^{2} k_{I} \chi_{\mu_{I}} y_{I}^{3} + \beta k_{E}^{2} \mu_{I} \chi_{k_{I}} y_{I}^{3} \nonumber\\
&+ \beta k_{E}^{2} \chi_{k_{I}} y_{I}^{2} \dfrac{\mathrm{d}y_{I}}{\mathrm{d}t} + \beta k_{E} k_{I}^{2} \chi_{\delta} y_{E} y_{I}^{2} + c \upsilon k_{E}^{2} k_{I}^{2} \chi_{\beta} y_{I}^{2} + \delta k_{E} k_{I}^{3} \chi_{\mu_{S}} y_{E} y_{I} \nonumber\\
&+ \delta k_{E} k_{I}^{2} \mu_{S} \chi_{k_{I}} y_{E} y_{I} + \delta k_{E} k_{I}^{2} \chi_{\beta} y_{E} y_{I}^{2} - \delta k_{E} k_{I}^{2} \chi_{k_{I}} y_{E} \dfrac{\mathrm{d}y_{I}}{\mathrm{d}t} + \delta k_{E} k_{I}^{2} \chi_{k_{I}} y_{I} \dfrac{\mathrm{d}y_{E}}{\mathrm{d}t} \nonumber\\
&- \delta k_{I}^{3} \mu_{S} \chi_{k_{E}} y_{E} y_{I} + \delta k_{I}^{3} \chi_{k_{E}} y_{E} \dfrac{\mathrm{d}y_{I}}{\mathrm{d}t} - \delta k_{I}^{3} \chi_{k_{E}} y_{I} \dfrac{\mathrm{d}y_{E}}{\mathrm{d}t} - k_{E}^{2} k_{I}^{2} \mu_{I} \chi_{\mu_{S}} y_{I}^{2} \nonumber\\
&- k_{E}^{2} k_{I}^{2} \mu_{S} \chi_{\mu_{I}} y_{I}^{2} - k_{E}^{2} k_{I}^{2} \chi_{\mu_{S}} y_{I} \dfrac{\mathrm{d}y_{I}}{\mathrm{d}t} - k_{E}^{2} k_{I} \mu_{I} \chi_{\beta} y_{I}^{3} - k_{E}^{2} k_{I} \chi_{\beta} y_{I}^{2} \dfrac{\mathrm{d}y_{I}}{\mathrm{d}t} \nonumber\\
&+ k_{E} k_{I}^{3} \mu_{S} \chi_{\delta} y_{E} y_{I} - k_{E} k_{I}^{3} \chi_{\delta} y_{E} \dfrac{\mathrm{d}y_{I}}{\mathrm{d}t} + k_{E} k_{I}^{3} \chi_{\delta} y_{I} \dfrac{\mathrm{d}y_{E}}{\mathrm{d}t},\label{eq:lin_sym_SEI_1}\\
0 =&\textcolor{white}{-}\delta k_{I}^{2} \chi_{\upsilon} y_{E} - \upsilon^{2} k_{E} k_{I} \chi_{\mu_{I}} y_{I} + \upsilon^{2} k_{E} \mu_{I} \chi_{k_{I}} y_{I} + \upsilon^{2} k_{E} \chi_{k_{I}} \dfrac{\mathrm{d}y_{I}}{\mathrm{d}t} - \upsilon^{2} k_{I}^{2} \chi_{\mu_{E}} y_{E} \nonumber\\
&- \upsilon^{2} k_{I} \mu_{I} \chi_{k_{E}} y_{I} - \upsilon^{2} k_{I} \chi_{k_{E}} \dfrac{\mathrm{d}y_{I}}{\mathrm{d}t} + \upsilon k_{E} k_{I} \chi_{\mu_{I}} y_{I} - \upsilon k_{E} \mu_{I} \chi_{k_{I}} y_{I} \nonumber\\
&- \upsilon k_{E} \chi_{k_{I}} \dfrac{\mathrm{d}y_{I}}{\mathrm{d}t} - \upsilon k_{I}^{2} \chi_{\delta} y_{E} + \upsilon k_{I} \mu_{I} \chi_{k_{E}} y_{I} + \upsilon k_{I} \chi_{k_{E}} \dfrac{\mathrm{d}y_{I}}{\mathrm{d}t} - k_{E} k_{I} \mu_{I} \chi_{\upsilon} y_{I}\nonumber\\
&- k_{E} k_{I} \chi_{\upsilon} \dfrac{\mathrm{d}y_{I}}{\mathrm{d}t}.\label{eq:lin_sym_SEI_2}
\end{align}
The family of generating vector fields solving these linearised symmetry conditions is given by
\begin{equation} 
\begin{split} X_{\boldsymbol{\theta}}=&\,\,[(\alpha_{2}-\alpha_{1})c\upsilon-\alpha_{2}c]\partial_{c}+\alpha_{1}\beta\partial_{\beta}+(\alpha_{2}-\alpha_{1})\delta\partial_{\delta}+(\alpha_{1}-\alpha_{2})\upsilon \left(\upsilon - 1\right)\partial_{\upsilon}\\
    &\,\,\,\,\,-(\alpha_{2}-\alpha_{1})\delta\partial_{\mu_{E}}+\alpha_{1}k_{I}\partial_{k_{I}}+\alpha_{2}k_{E}\partial_{k_{E}}\,,
    \end{split}
    \label{eq:X_SEI}
\end{equation}
where $\alpha_{1}$ and $\alpha_{2}$ are two arbitrary coefficients. Using this family of vector fields, we generate parameter symmetries of the SEI model. Specifically, these symmetries $\Gamma_{\varepsilon}^{\boldsymbol{\theta}}$ generate transformed parameter vectors according to $\Gamma_{\varepsilon}^{\boldsymbol{\theta}}:\boldsymbol{\theta}\mapsto\hat{\boldsymbol{\theta}}(\varepsilon)$ that are given by
\begin{equation}
\hat{\boldsymbol{\theta}}(\varepsilon)=(\hat{c}(\varepsilon),\hat{\beta}(\varepsilon),\hat{\delta}(\varepsilon),\hat{\upsilon}(\varepsilon),\hat{\mu}_{S}(\varepsilon),\hat{\mu}_{E}(\varepsilon),\hat{\mu}_{I}(\varepsilon),\hat{k}_{I}(\varepsilon),\hat{k}_{E}(\varepsilon)).
\label{eq:theta_SEI_trans}
\end{equation}
Moreover, the transformed parameters $\hat{\boldsymbol{\theta}}(\varepsilon)$ solve the following system of ODEs:
\begin{align}
  \frac{\mathrm{d}\hat{c}}{\mathrm{d}\varepsilon}&=(\alpha_{2}-\alpha_{1})\hat{c}\hat{\upsilon}-\alpha_{2}\hat{c},\quad&\hat{c}(\varepsilon=0)=c,\label{eq:trans_c}\\
  \frac{\mathrm{d}\hat{\beta}}{\mathrm{d}\varepsilon}&=\alpha_{1}\hat{\beta},\quad&\hat{\beta}(\varepsilon=0)=\beta,\label{eq:trans_beta}\\
  \frac{\mathrm{d}\hat{\delta}}{\mathrm{d}\varepsilon}&=(\alpha_{2}-\alpha_{1})\hat{\delta},\quad&\hat{\delta}(\varepsilon=0)=\delta,\label{eq:trans_delta}\\
  \frac{\mathrm{d}\hat{\upsilon}}{\mathrm{d}\varepsilon}&=(\alpha_{1}-\alpha_{2})\hat{\upsilon} \left(\hat{\upsilon} - 1\right),\quad&\hat{\upsilon}(\varepsilon=0)=\upsilon,\label{eq:trans_epsilon}\\
  \frac{\mathrm{d}\hat{\mu}_{S}}{\mathrm{d}\varepsilon}&=0,\quad&\hat{\mu}_{S}(\varepsilon=0)=\mu_{S},\label{eq:trans_muS}\\
  \frac{\mathrm{d}\hat{\mu}_{E}}{\mathrm{d}\varepsilon}&=-(\alpha_{2}-\alpha_{1})\hat{\delta},\quad&\hat{\mu}_{E}(\varepsilon=0)=\mu_{E},\label{eq:trans_muE}\\
  \frac{\mathrm{d}\hat{\mu}_{I}}{\mathrm{d}\varepsilon}&=0,\quad&\hat{\mu}_{I}(\varepsilon=0)=\mu_{I},\label{eq:trans_muI}\\
  \frac{\mathrm{d}\hat{k}_{E}}{\mathrm{d}\varepsilon}&=\alpha_{2}\hat{k}_{E},\quad&\hat{k}_{E}(\varepsilon=0)=k_{E},\label{eq:trans_kE}\\
  \frac{\mathrm{d}\hat{k}_{I}}{\mathrm{d}\varepsilon}&=\alpha_{1}\hat{k}_{I},\quad&\hat{k}_{I}(\varepsilon=0)=k_{I}.\label{eq:trans_kI}  
\end{align}
This system of ODEs is readily solved numerically in order to characterise the action of any specific symmetry defined by specific choices of the coefficients $\alpha_{1}$ and $\alpha_{2}$. Since the parameter space of the SEI model is nine-dimensional, we visualise the action of the parameter symmetry of interest in two- and three-dimensional subspaces. Specifically, we have numerically generated six parameter vectors as starting points illustrated by diamonds in order to plot the corresponding transformed parameter vectors $\hat{\boldsymbol{\theta}}(\varepsilon)$ in Eq.~\eqref{eq:theta_SEI_trans}, and relevant figures are presented in the next step of the CaLinInv recipe. The interpretation of these curves in parameter space is that the illustrated parameters have indistinguishable model behaviour.    


\subsubsection*{Step 3: Universal Invariants}

The following seven universal parameter invariants were found (for details, see the Appendices) using the method of characteristics:
  
  \begin{align}
I_{1}&=\mu_{I}\,,\label{eq:SEI_I_1}\\
      I_{2}&=\mu_{S}\,,\label{eq:SEI_I_2}\\
      I_{3}&=\delta+\mu_{E}\,,  \label{eq:I3}\\
      I_{4}&=\frac{\beta}{k_{I}}\,,\label{eq:I4}\\
      I_{5}&=\frac{\upsilon}{\delta(1-\upsilon)}\,,\label{eq:I5}\\
      I_{6}&=\beta{c}\upsilon\,,\label{eq:I6}\\
      I_{7}&=\frac{\beta\delta}{k_{E}}\,.\label{eq:I7}      
    \end{align}
These universal parameter invariants correspond to the same globally structurally identifiable parameter combinations found by Renardy et al.~\cite{renardy2022structural} using the standard differential algebra approach. Since these universal parameter invariants also correspond to the locally structurally identifiable parameter combinations (Cor. \ref{cor:structural_identifiability_symmetries}), the locally and globally structurally identifiable parameter combinations are the same in the case of the SEI model. By solving the ODE system in Eqs.~\eqref{eq:trans_c}--\eqref{eq:trans_kI} for the arbitrary coefficients $(\alpha_{1},\alpha_{2})=(2,1)$, we visualise the action of the corresponding parameter symmetry and illustrate that it preserves all seven universal parameter invariants.

\begin{figure}[htbp!]
\begin{center}
\includegraphics[width=\textwidth]{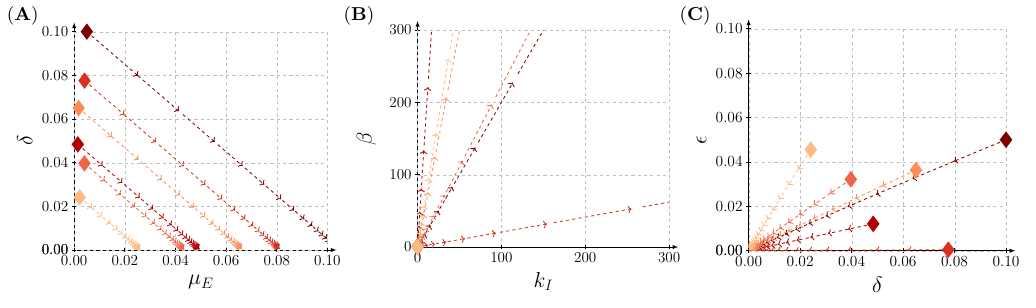}
  \caption{\textit{Two-dimensional projections of the action of a parameter symmetry $\Gamma_{\varepsilon}^{\boldsymbol{\theta}}$ of the SEI model}. The action of the symmetry $\Gamma_{\varepsilon}^{\boldsymbol{\theta}}$ generated by solving the system of ODEs in Eqs.~\eqref{eq:trans_c}-\eqref{eq:trans_kI} for the coefficients $(\alpha_{1},\alpha_{2})=(2,1)$ on six parameters illustrated by diamonds is visualised in three different two-dimensional subspaces of the nine-dimensional parameter space of the SEI model. The interpretation of these curves in parameter space is that the illustrated parameters have indistinguishable model behaviour. These subspaces are (\textbf{A}) $(\mu_{E},\delta)$ for which the invariant $I_{3}=\delta+\mu_{E}$ is preserved, (\textbf{B}) $(k_{I},\beta)$ for which the invariant $I_{4}=\beta/k_{I}$ is preserved and (\textbf{C}) $(\delta,\upsilon)$ for which the invariant $I_{5}=\upsilon/(\delta(1-\upsilon))$ is preserved.}
  \label{fig:SEI_2D}
\end{center}
\end{figure}

First, we visualise the action of this symmetry on three parameter pairs (Fig.~\ref{fig:SEI_2D}). These three parameter pairs are $(\mu_{E},\delta)$ for which the symmetry preserves the invariant $I_{3}$ in Eq.~\eqref{eq:I3}, $(k_{I},\beta)$ for which the symmetry preserves the invariant $I_{4}$ in Eq.~\eqref{eq:I4} and $(\delta,\upsilon)$ for which the symmetry preserves the invariant $I_{5}$ in Eq.~\eqref{eq:I5}. Similarly, we visualise the action of this symmetry on two parameter triplets (Fig.~\ref{fig:SEI_3D}). These triplets are $(\beta,\upsilon,c)$ for which the symmetry preserves the invariant $I_{6}$ in Eq.~\eqref{eq:I6}. and $(\beta,\delta,k_{E})$ for which the symmetry preserves the invariant $I_{7}$ in Eq.~\eqref{eq:I7}.

\begin{figure}[htbp!]
\begin{center}
\includegraphics[width=\textwidth]{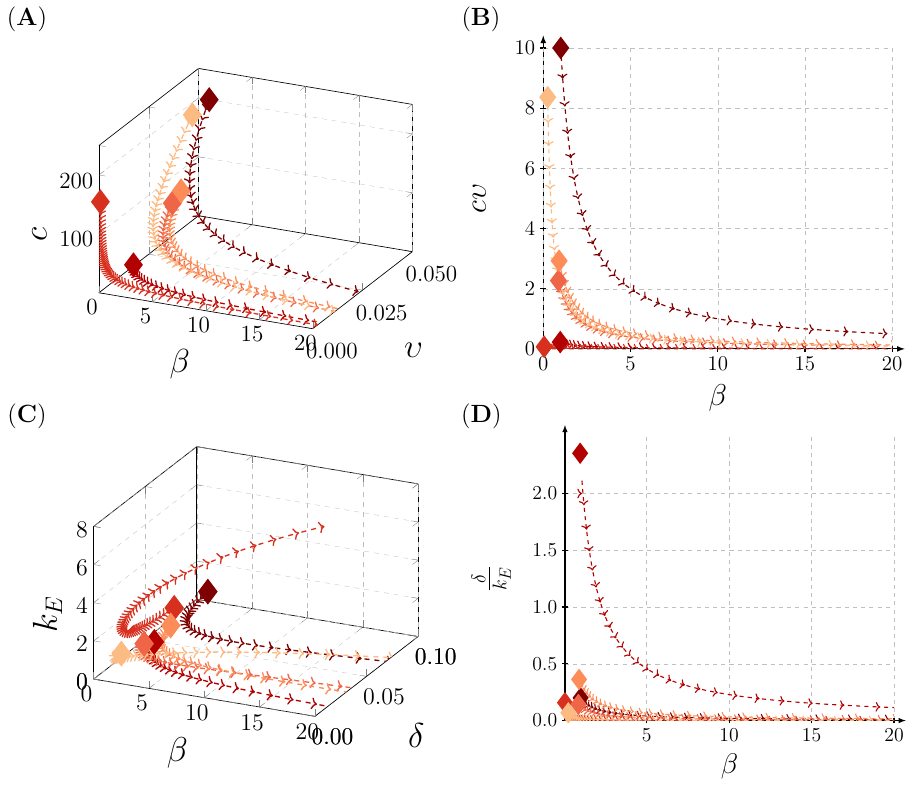}
  \caption{\textit{Three-dimensional projections of the action of a parameter symmetry $\Gamma_{\varepsilon}^{\boldsymbol{\theta}}$ of the SEI model}. The interpretation of these curves in parameter space is that the illustrated parameters have indistinguishable model behaviour. The action of the symmetry $\Gamma_{\varepsilon}^{\boldsymbol{\theta}}$ generated by solving the system of ODEs in Eqs.~\eqref{eq:trans_c}--\eqref{eq:trans_kI} for the coefficients $(\alpha_{1},\alpha_{2})=(2,1)$ on six parameters illustrated by diamonds is visualised in two different three-dimensional subspaces as well as in two different two-dimensional subspaces of the nine-dimensional parameter space of the SEI model. These subspaces are (\textbf{A}) $(\beta,\upsilon,c)$ for which the invariant $I_{6}=\beta c\upsilon$ is preserved, (\textbf{B}) $(\beta,c\upsilon)$ for which the invariant $I_{6}=\beta c\upsilon$ is preserved, (\textbf{C}) $(\beta,\delta,k_{E})$ for which the invariant $I_{7}=\beta\delta/k_{E}$ is preserved and (\textbf{D}) $(\beta,\delta/k_{E})$ for which the invariant $I_{7}=\beta\delta/k_{E}$ is preserved. }
  \label{fig:SEI_3D}
\end{center}
\end{figure}


\section{Discussion}

In this work, we have demonstrated how local structural identifiability can be understood in terms of the differential invariants of parameter symmetries. For the last two decades, the notion of classical Lie symmetries of ODEs acting on the independent and dependent variables by mapping solutions to other solutions~\cite{bluman1989symmetries,hydon2000symmetry,olver2000applications,stephani1989differential} has been extended to full symmetries which also account for parameters. Such full symmetries have been a large focus of research on the structural identifiability of mechanistic ODE models~\cite{yates2009structural,merkt2015higher,massonis2020finding,castro2020structuralIdentifiability,villaverde2022symmetries}, and in particular a large emphasis has been put on developing algorithms for finding such symmetries in an automated fashion. However, the link between algebraic methods for determining structural identifiability and symmetry-based methods has, until this point, remained elusive. In this work, we establish this conceptual link by introducing so-called \textit{parameter symmetries}, Lie transformations that alter parameters while simultaneously preserving the observed outputs. In addition, we demonstrate that local structural identifiability can be understood in terms of the differential invariants of these parameter symmetries. Based on these results, we propose a three step recipe referred to as the CaLinInv recipe which involves: (i) re-writing the original first-order ODE system as an equivalent ODE system for the outputs, also referred to as the \textit{Ca}nonical coordinates; (ii) finding the parameter symmetries by solving the \textit{Lin}earised symmetry conditions; and (iii) elucidating the global structural identifiability by calculating the differential \textit{Inv}ariants of the parameter symmetries. We later demonstrated practical use of the CaLinInv recipe by analysing the local structural identifiability of two previously analysed mechanistic models of biological systems. 

The CaLinInv recipe constitutes a new framing of the classical differential algebra approach for elucidating structural identifiability in terms of Lie symmetries. The steps in this recipe are reminiscent of the differential algebra approach for structural identifiability. In fact, the first steps---finding algebraic equations relating inputs and outputs with parameters~\cite{ljung1994global}---are identical. Technically, the differential algebra approach constructs a map between the parameters and the parameter combinations that can be inferred from the inputs and outputs, and then structural identifiability requires that this map is injective~\cite{rey2023benchmarking}. The parameter symmetries proposed in this work are a way of constructing such maps, and the injectivity criterion can be understood in terms of the universal differential invariants of parameter symmetries. By framing structural identifiability in terms of universal invariants of parameter symmetries, we understand why the standard differential algebra approach, which extracts coefficients in front of the monomials of the polynomial system of output ODEs, always finds identifiable parameter quantities; that is, universal parameter invariants. This is due to the fact that the coefficients that are extracted in the differential algebra approach will either be a constant or a universal parameter invariant. This property is ensured by the definition of invariants of symmetries combined with the so-called \textit{linearised symmetry conditions}, the equations defining parameter symmetries. In other words, the standard differential algebra approach is consistent with the notion of local structural identifiability expressed in terms of universal parameter invariants. Moreover, our symmetry-based approach for analysing local structural identifiability is theoretically generalisable to other mechanistic models consisting of, say, systems of partial differential equations. However, in practice, computer-assisted versions of this approach must be developed in order to analyse the structural identifiability of such systems. In particular, two aspects of the symmetry-based structural identifiability analysis should be automated, namely the formulation of the input-output system and an implementation for solving the linearised symmetry conditions.

An interesting future research direction is to automate the CaLinInv recipe for systems of ODEs and eventually systems of partial differential equations. In the context of ODEs, it is known that if the right-hand sides of the original first-order system are rational functions of the states it is always possible to re-write the original system of first-order ODEs as a system of ODEs depending solely on the observed outputs~\cite{ljung1994global}. Given such a re-formulated system in terms of the observed outputs, the two remaining steps of the CaLinInv recipe are straightforward to automate using symbolic calculations. This is also why the recipe can be automated, since many existing software for structural identifiability (e.g.~\cite{dong2023structuralIdentifiabilityJL}) conduct the first step of re-writing the original system so that it solely depends on the observed outputs in an automated fashion. Accordingly, the CaLinInv recipe can be implemented on top of existing algorithms for structural identifiability analyses based on the differential algebra approach, which would result in an algorithm that not only yields the identifiable parameter combinations but also the family of parameter transformations that preserves the observed outputs; that is, a family of parameter symmetries. 

Altogether, this work establishes a link between the existing body of work on full symmetries~\cite{yates2009structural,merkt2015higher,massonis2020finding,castro2020structuralIdentifiability,villaverde2022symmetries} and the differential algebra approach for structural identifiability~\cite{ljung1994global,hong2020global,saccomani2001new,walter1982global}. Until now, it has been unclear how the full symmetries transforming independent and dependent variables as well as parameters relate to structural identifiability. The result which is closest to such a link was presented by Castro and de Boer~\cite{castro2020structuralIdentifiability} which states that a particular parameter is globally structurally identifiable if the only way to scale this parameter by a scaling factor that preserves the observed outputs, is if this scaling factor equals one. In fact, this is exactly what it means to say that the particular parameter of interest is a parameter invariant, and our theoretical framework based on parameter symmetries has formalised this result by demonstrating that a parameter is locally structurally identifiable if and only if it is a universal parameter invariant. Indeed, our result generalises to any parameter symmetry as it is not restricted to the scalings studied by Castro and de Boer~\cite{castro2020structuralIdentifiability}. As has been noted previously, by only focusing on scalings one misses other important parameter transformations that preserve observed outputs~\cite{villaverde2021testing}, and thus the conclusions about structural identifiability drawn from studying such scalings can be misleading. In addition, Lie symmetries acting on states and parameters have previously been calculated in an automated fashion~\cite{massonis2020finding}, but the exact link between structural identifiability and these calculated symmetries has been unclear. A succinct way of expressing this link in words which is our main result is that the \textit{locally structurally identifiable parameter quantities are given by universal parameter invariants}.  We have made a case for a perspective in which local structural identifiability is expressed in terms of differential invariants of parameter symmetries, and this work is a stepping stone towards fully exploiting the power of symmetry methods within the realm of local structural identifiability.



\section*{Data availability statement}
Regarding the structural identifiability-analysis of the SEI-model, details and relevant scripts are available at the public github-repository associated with this project; \url{https://github.com/JohannesBorgqvist/symmetries_and_structural_identifiability}.


\section*{Acknowledgements}
JGB is funded by a grant from the Wenner-Gren foundations (Grant number: FT2023-0005). JBG thanks the Wenner-Gren foundations for a research fellowship. APB thanks the Mathematical Institute, Oxford for a Hooke Research Fellowship. This work was supported by a grant from the Simons Foundation (MP-structural identifiabilityP-00001828, REB). or the purpose of open access, the author has applied a
CC BY public copyright licence to any author accepted manuscript arising from this submission.

\section*{CRediT author statment}
\begin{itemize}
\item[\textbf{JGB}] Conceptualization, Methodology, Visualization (made the figures), Writing - Original Draft, Writing - Review \& Editing, Formal analysis (derived and proved theorems and conducted calculations), Software (wrote scripts that conducted calculations). 
\item[\textbf{APB}] Conceptualization, Writing - Original Draft, Writing - Review \& Editing.
\item[\textbf{FO}] Conceptualization, Writing - Original Draft, Writing - Review \& Editing,
\item[\textbf{REB}] Conceptualization, Writing - Original Draft, Writing - Review \& Editing.
\end{itemize}

\clearpage
\appendix  

\renewcommand{\thesection}{\Alph{section}}

\counterwithin{equation}{section}  

\counterwithin{figure}{section}
\counterwithin{table}{section}

\section*{Appendices}

\section{Analysing structural identifiability of an epidemiological model}

We study an SEI model of the epidemiology of tuberculosis~\cite{renardy2022structural}. This model consists of three states: $S(t)$ corresponds to the susceptible population; $E(t)$ corresponds to the exposed population; and $I(t)$ corresponds to the infected population. Moreover, this model has seven rate parameters: $c$ is the birth rate; $\beta$ is the transmission rate; $\upsilon$ is the probability of primary infection; $\delta$ is the reactivation rate and $\mu_{S}$, $\mu_{E}$ and $\mu_{I}$ are death rates. All parameters are assumed to be positive. The corresponding system of ODEs is given by
\begin{equation}
    \begin{split}
\dfrac{\mathrm{d}S}{\mathrm{d}t}&=c-\beta SI-\mu_{S}S,\\
\dfrac{\mathrm{d}E}{\mathrm{d}t}&=(1-\upsilon)\beta SI-\delta E-\mu_{E}E,\\
\dfrac{\mathrm{d}I}{\mathrm{d}t}&=\upsilon\beta SI+\delta E-\mu_{I}I.\\
\end{split}
    \label{eq:SEI_app}
\end{equation}
Provided this system, we consider two outputs given by a proportion of the exposed population and a proportion of the infected population where these proportions are encoded by the parameters $k_{E}$ and $k_{I}$, respectively. Then, the two observed outputs denoted by $y_{E}$ and $y_{I}$ are given by
\begin{align}
    y_{E}&=k_{E}E,\label{eq:output_E_app}\\
    y_{I}&=k_{I}I.\label{eq:output_I_app}
\end{align}
In total, the system has nine rate parameters so that
\begin{equation}
\boldsymbol{\theta}=(c,\beta,\delta,\upsilon,\mu_{S},\mu_{E},\mu_{I},k_{E},k_{I}).
\label{eq:theta_SEI_app}
\end{equation}
We implement the CaLinInv recipe in order to analyse global structural identifiability. Importantly, we compare the outcomes of these calculations to those obtained through the standard differential algebra approach. Details and relevant scripts are  available at the public github-repository associated with this project; \url{https://github.com/JohannesBorgqvist/symmetries_and_structural_identifiability}.


\subsection{Finding generating vector fields by solving the linearised symmetry conditions}

Starting with the first step in the CaLinInv recipe, the system of output ODEs is given by
\begin{align}
\dfrac{\mathrm{d}y_{E}}{\mathrm{d}t} =&-\frac{\delta y_{E}}{\upsilon} - \frac{k_{E} \mu_{I} y_{I}}{k_{I}} - \frac{k_{E}}{k_{I}}\dfrac{\mathrm{d}y_{I}}{\mathrm{d}t} - \mu_{E} y_{E} + \frac{k_{E} \mu_{I} y_{I}}{\upsilon k_{I}} + \frac{k_{E}}{\upsilon k_{I}}\dfrac{\mathrm{d}y_{I}}{\mathrm{d}t},\label{eq:ODE_y_E_2_app}\\  
\dfrac{\mathrm{d}^{2}y_{I}}{\mathrm{d}t^{2}} =&\textcolor{white}{-}\beta c \upsilon y_{I} + \frac{\beta \delta y_{E} y_{I}}{k_{E}} - \frac{\beta \mu_{I} y_{I}^{2}}{k_{I}} - \frac{\beta y_{I}}{k_{I}}\dfrac{\mathrm{d}y_{I}}{\mathrm{d}t} + \frac{\delta k_{I} \mu_{S} y_{E}}{k_{E}} - \frac{\delta k_{I} y_{E}}{k_{E} y_{I}}\dfrac{\mathrm{d}y_{I}}{\mathrm{d}t} + \frac{\delta k_{I}}{k_{E}}\dfrac{\mathrm{d}y_{E}}{\mathrm{d}t}\nonumber\\
&-\mu_{I} \mu_{S} y_{I} - \mu_{S} \dfrac{\mathrm{d}y_{I}}{\mathrm{d}t} + \frac{1}{y_{I}}\left(\dfrac{\mathrm{d}y_{I}}{\mathrm{d}t}\right)^{2}.\label{eq:ODE_y_I_2_app}
\end{align}
Moreover, the family of generating vector fields associated with the parameter symmetries of interest has the following structure
\begin{equation}
X_{\boldsymbol{\theta}}=\chi_{c}\partial_{c}+\chi_{\beta}\partial_{\beta}+\chi_{\delta}\partial_{\delta}+\chi_{\upsilon}\partial_{\upsilon}+\chi_{\mu_S}\partial_{\mu_S}+\chi_{\mu_E}\partial_{\mu_E}+\chi_{\mu_I}\partial_{\mu_I}+\chi_{k_{E}}\partial_{k_{E}}+\chi_{k_{I}}\partial_{k_{I}},
\label{eq:generator_SEI}
\end{equation}
where all infinitesimals are functions of the rate parameters $\boldsymbol{\theta}$ in Eq.~\eqref{eq:theta_SEI_app}. The two linearised symmetry conditions defining the generating vector fields can be simplified to the following two equations
\begin{align}
0 =&\textcolor{white}{-}\beta c k_{E}^{2} k_{I}^{2} \chi_{\upsilon} y_{I}^{2} - \beta \delta k_{I}^{2} \chi_{k_{E}} y_{E} y_{I}^{2} + \beta \upsilon k_{E}^{2} k_{I}^{2} \chi_{c} y_{I}^{2} - \beta k_{E}^{2} k_{I} \chi_{\mu_{I}} y_{I}^{3} + \beta k_{E}^{2} \mu_{I} \chi_{k_{I}} y_{I}^{3} \nonumber\\
&+ \beta k_{E}^{2} \chi_{k_{I}} y_{I}^{2} \dfrac{\mathrm{d}y_{I}}{\mathrm{d}t} + \beta k_{E} k_{I}^{2} \chi_{\delta} y_{E} y_{I}^{2} + c \upsilon k_{E}^{2} k_{I}^{2} \chi_{\beta} y_{I}^{2} + \delta k_{E} k_{I}^{3} \chi_{\mu_{S}} y_{E} y_{I} \nonumber\\
&+ \delta k_{E} k_{I}^{2} \mu_{S} \chi_{k_{I}} y_{E} y_{I} + \delta k_{E} k_{I}^{2} \chi_{\beta} y_{E} y_{I}^{2} - \delta k_{E} k_{I}^{2} \chi_{k_{I}} y_{E} \dfrac{\mathrm{d}y_{I}}{\mathrm{d}t} + \delta k_{E} k_{I}^{2} \chi_{k_{I}} y_{I} \dfrac{\mathrm{d}y_{E}}{\mathrm{d}t} \nonumber\\
&- \delta k_{I}^{3} \mu_{S} \chi_{k_{E}} y_{E} y_{I} + \delta k_{I}^{3} \chi_{k_{E}} y_{E} \dfrac{\mathrm{d}y_{I}}{\mathrm{d}t} - \delta k_{I}^{3} \chi_{k_{E}} y_{I} \dfrac{\mathrm{d}y_{E}}{\mathrm{d}t} - k_{E}^{2} k_{I}^{2} \mu_{I} \chi_{\mu_{S}} y_{I}^{2} \nonumber\\
&- k_{E}^{2} k_{I}^{2} \mu_{S} \chi_{\mu_{I}} y_{I}^{2} - k_{E}^{2} k_{I}^{2} \chi_{\mu_{S}} y_{I} \dfrac{\mathrm{d}y_{I}}{\mathrm{d}t} - k_{E}^{2} k_{I} \mu_{I} \chi_{\beta} y_{I}^{3} - k_{E}^{2} k_{I} \chi_{\beta} y_{I}^{2} \dfrac{\mathrm{d}y_{I}}{\mathrm{d}t} \nonumber\\
&+ k_{E} k_{I}^{3} \mu_{S} \chi_{\delta} y_{E} y_{I} - k_{E} k_{I}^{3} \chi_{\delta} y_{E} \dfrac{\mathrm{d}y_{I}}{\mathrm{d}t} + k_{E} k_{I}^{3} \chi_{\delta} y_{I} \dfrac{\mathrm{d}y_{E}}{\mathrm{d}t},\label{eq:lin_sym_SEI_1_app}\\
0 =&\textcolor{white}{-}\delta k_{I}^{2} \chi_{\upsilon} y_{E} - \upsilon^{2} k_{E} k_{I} \chi_{\mu_{I}} y_{I} + \upsilon^{2} k_{E} \mu_{I} \chi_{k_{I}} y_{I} + \upsilon^{2} k_{E} \chi_{k_{I}} \dfrac{\mathrm{d}y_{I}}{\mathrm{d}t} - \upsilon^{2} k_{I}^{2} \chi_{\mu_{E}} y_{E} \nonumber\\
&- \upsilon^{2} k_{I} \mu_{I} \chi_{k_{E}} y_{I} - \upsilon^{2} k_{I} \chi_{k_{E}} \dfrac{\mathrm{d}y_{I}}{\mathrm{d}t} + \upsilon k_{E} k_{I} \chi_{\mu_{I}} y_{I} - \upsilon k_{E} \mu_{I} \chi_{k_{I}} y_{I} \nonumber\\
&- \upsilon k_{E} \chi_{k_{I}} \dfrac{\mathrm{d}y_{I}}{\mathrm{d}t} - \upsilon k_{I}^{2} \chi_{\delta} y_{E} + \upsilon k_{I} \mu_{I} \chi_{k_{E}} y_{I} + \upsilon k_{I} \chi_{k_{E}} \dfrac{\mathrm{d}y_{I}}{\mathrm{d}t} - k_{E} k_{I} \mu_{I} \chi_{\upsilon} y_{I}\nonumber\\
&- k_{E} k_{I} \chi_{\upsilon} \dfrac{\mathrm{d}y_{I}}{\mathrm{d}t}.\label{eq:lin_sym_SEI_2_app}
\end{align}
The first linearised symmetry condition in Eq.~\eqref{eq:lin_sym_SEI_1_app} decomposes into subequations based on the products between the various powers of the states and their derivatives according to
\begin{align}
y_{E} \dfrac{\mathrm{d}y_{I}}{\mathrm{d}t}\quad:\quad{0}&=- \delta k_{E} k_{I}^{2} \chi_{k_{I}} + \delta k_{I}^{3} \chi_{k_{E}} - k_{E} k_{I}^{3} \chi_{\delta},\label{eq:det_eq_1}\\
y_{I} \dfrac{\mathrm{d}y_{E}}{\mathrm{d}t}\quad:\quad{0}&=\delta k_{E} k_{I}^{2} \chi_{k_{I}} - \delta k_{I}^{3} \chi_{k_{E}} + k_{E} k_{I}^{3} \chi_{\delta},\label{eq:det_eq_2}\\
y_{I} \dfrac{\mathrm{d}y_{I}}{\mathrm{d}t}\quad:\quad{0}&=k_{E}^{2} k_{I}^{2} \chi_{\mu_{S}},\label{eq:det_eq_3}\\
y_{E} y_{I}\quad:\quad{0}&=\delta k_{E} k_{I}^{3} \chi_{\mu_{S}} + \delta k_{E} k_{I}^{2} \mu_{S} \chi_{k_{I}} - \delta k_{I}^{3} \mu_{S} \chi_{k_{E}} + k_{E} k_{I}^{3} \mu_{S} \chi_{\delta},\label{eq:det_eq_4}\\
y_{I}^{2}\quad:\quad{0}&=\beta c k_{E}^{2} k_{I}^{2} \chi_{\upsilon} + \beta \upsilon k_{E}^{2} k_{I}^{2} \chi_{c} + c \upsilon k_{E}^{2} k_{I}^{2} \chi_{\beta} - k_{E}^{2} k_{I}^{2} \mu_{I} \chi_{\mu_{S}} - k_{E}^{2} k_{I}^{2} \mu_{S} \chi_{\mu_{I}},\label{eq:det_eq_5}\\
y_{I}^{2} \dfrac{\mathrm{d}y_{I}}{\mathrm{d}t}\quad:\quad{0}&=\beta k_{E}^{2} \chi_{k_{I}} - k_{E}^{2} k_{I} \chi_{\beta},\label{eq:det_eq_6}\\
y_{E} y_{I}^{2}\quad:\quad{0}&=- \beta \delta k_{I}^{2} \chi_{k_{E}} + \beta k_{E} k_{I}^{2} \chi_{\delta} + \delta k_{E} k_{I}^{2} \chi_{\beta},\label{eq:det_eq_7}\\
y_{I}^{3}\quad:\quad{0}&=- \beta k_{E}^{2} k_{I} \chi_{\mu_{I}} + \beta k_{E}^{2} \mu_{I} \chi_{k_{I}} - k_{E}^{2} k_{I} \mu_{I} \chi_{\beta}.\label{eq:det_eq_8}
\end{align}
Similarly, the second linearised symmetry condition in Eq.~\eqref{eq:lin_sym_SEI_2_app} decomposes into
\begin{align}
\dfrac{\mathrm{d}y_{I}}{\mathrm{d}t}\quad:\quad{0}&=\upsilon^{2} k_{E} \chi_{k_{I}} - \upsilon^{2} k_{I} \chi_{k_{E}} - \upsilon k_{E} \chi_{k_{I}} + \upsilon k_{I} \chi_{k_{E}} - k_{E} k_{I} \chi_{\upsilon},\quad\quad\;\;\label{eq:det_eq_9}\\
y_{E}\quad:\quad{0}&=\delta k_{I}^{2} \chi_{\upsilon} - \upsilon^{2} k_{I}^{2} \chi_{\mu_{E}} - \upsilon k_{I}^{2} \chi_{\delta},\label{eq:det_eq_10}\\
y_{I}\quad:\quad{0}&=- \upsilon^{2} k_{E} k_{I} \chi_{\mu_{I}} + \upsilon^{2} k_{E} \mu_{I} \chi_{k_{I}} - \upsilon^{2} k_{I} \mu_{I} \chi_{k_{E}}\nonumber\\
&\quad+ \upsilon k_{E} k_{I} \chi_{\mu_{I}} - \upsilon k_{E} \mu_{I} \chi_{k_{I}} + \upsilon k_{I} \mu_{I} \chi_{k_{E}} - k_{E} k_{I} \mu_{I} \chi_{\upsilon}.\label{eq:det_eq_11}
\end{align}

These equations constitute a system of linear equations on the form $M\boldsymbol{\chi}=\mathbf{0}$ where $\boldsymbol{\chi}=(\chi_{c},\chi_{\beta},\chi_{\delta},\chi_{\upsilon},\chi_{\mu_{S}},\chi_{\mu_{E}},\chi_{\mu_{I}},\chi_{k_{I}},\chi_{k_{E}})\in\mathbb{R}^{9}$ contains the parameter infinitesimals and where the $11\times 9$ matrix $M$ is given by

{\tiny
\begin{equation}
M=\begin{pmatrix}0 & 0 & - k_{E} k_{I}^{3} & 0 & 0 & 0 & 0 & - \delta k_{E} k_{I}^{2} & \delta k_{I}^{3}\\0 & 0 & k_{E} k_{I}^{3} & 0 & 0 & 0 & 0 & \delta k_{E} k_{I}^{2} & - \delta k_{I}^{3}\\0 & 0 & 0 & 0 & - k_{E}^{2} k_{I}^{2} & 0 & 0 & 0 & 0\\0 & 0 & k_{E} k_{I}^{3} \mu_{S} & 0 & \delta k_{E} k_{I}^{3} & 0 & 0 & \delta k_{E} k_{I}^{2} \mu_{S} & - \delta k_{I}^{3} \mu_{S}\\\beta \upsilon k_{E}^{2} k_{I}^{2} & c \upsilon k_{E}^{2} k_{I}^{2} & 0 & \beta c k_{E}^{2} k_{I}^{2} & - k_{E}^{2} k_{I}^{2} \mu_{I} & 0 & - k_{E}^{2} k_{I}^{2} \mu_{S} & 0 & 0\\0 & - k_{E}^{2} k_{I} & 0 & 0 & 0 & 0 & 0 & \beta k_{E}^{2} & 0\\0 & \delta k_{E} k_{I}^{2} & \beta k_{E} k_{I}^{2} & 0 & 0 & 0 & 0 & 0 & - \beta \delta k_{I}^{2}\\0 & - k_{E}^{2} k_{I} \mu_{I} & 0 & 0 & 0 & 0 & - \beta k_{E}^{2} k_{I} & \beta k_{E}^{2} \mu_{I} & 0\\0 & 0 & 0 & - k_{E} k_{I} & 0 & 0 & 0 & \upsilon^{2} k_{E} - \upsilon k_{E} & - \upsilon^{2} k_{I} + \upsilon k_{I}\\0 & 0 & - \upsilon k_{I}^{2} & \delta k_{I}^{2} & 0 & - \upsilon^{2} k_{I}^{2} & 0 & 0 & 0\\0 & 0 & 0 & - k_{E} k_{I} \mu_{I} & 0 & 0 & - \upsilon^{2} k_{E} k_{I} + \upsilon k_{E} k_{I} & \upsilon^{2} k_{E} \mu_{I} - \upsilon k_{E} \mu_{I} & - \upsilon^{2} k_{I} \mu_{I} + \upsilon k_{I} \mu_{I}\end{pmatrix}.
\end{equation}
}
Any solution $\boldsymbol{\chi}$ can be written as a linear combination of the basis vectors of the null space of the matrix $M$ denoted by $\mathcal{N}(M)$. This null space is two-dimensional and given by
\begin{equation}    
\mathcal{N}(M)\coloneqq\left\{\boldsymbol{\chi}=\begin{pmatrix}\chi_{c}\\\chi_{\beta}\\\chi_{\delta}\\\chi_{\upsilon}\\\chi_{\mu_{S}}\\\chi_{\mu_{E}}\\\chi_{\mu_{I}}\\\chi_{k_{I}}\\\chi_{k_{E}}\end{pmatrix}\in\mathbb{R}^{9}:\boldsymbol{\chi}\in\mathrm{Span}\left\{\begin{pmatrix}- c \upsilon\\\beta\\- \delta\\\upsilon \left(\upsilon - 1\right)\\0\\\delta\\0\\k_{I}\\0\end{pmatrix},\begin{pmatrix}c \left(\upsilon - 1\right)\\0\\\delta\\\upsilon \left(1 - \upsilon\right)\\0\\- \delta\\0\\0\\k_{E}\end{pmatrix}\right\}\right\}.
\end{equation}
As such, we consider the following parameter infinitesimals
\begin{equation}
\begin{pmatrix}\chi_{c}\\\chi_{\beta}\\\chi_{\delta}\\\chi_{\upsilon}\\\chi_{\mu_{S}}\\\chi_{\mu_{E}}\\\chi_{\mu_{I}}\\\chi_{k_{I}}\\\chi_{k_{E}}\end{pmatrix}=\alpha_{1}\begin{pmatrix}- c \upsilon\\\beta\\- \delta\\\upsilon \left(\upsilon - 1\right)\\0\\\delta\\0\\k_{I}\\0\end{pmatrix}+\alpha_{2}\begin{pmatrix}c \left(\upsilon - 1\right)\\0\\\delta\\\upsilon \left(1 - \upsilon\right)\\0\\- \delta\\0\\0\\k_{E}\end{pmatrix}=\begin{pmatrix}(\alpha_{2}-\alpha_{1})c\upsilon-c\alpha_{2}\\\alpha_{1}\beta \\(\alpha_{2}-\alpha_{1})\delta \\(\alpha_{1}-\alpha_{2})\upsilon \left(\upsilon - 1\right)\\0\\-(\alpha_{2}-\alpha_{1})\delta\\0\\\alpha_{1}k_{I}\\\alpha_{2}k_{E}\end{pmatrix},\label{eq:null_space_solution}
\end{equation}
which depend on two arbitrary coefficients, $\alpha_{1}$ and $\alpha_{2}$. Consequently, the infinitesimal generators of the family of parameter symmetries of the SEI model are given by
\begin{equation} 
\begin{split}
    X_{\boldsymbol{\theta}}=&\,\,[(\alpha_{2}-\alpha_{1})c\upsilon-\alpha_{2}c]\partial_{c}+\alpha_{1}\beta\partial_{\beta}+(\alpha_{2}-\alpha_{1})\delta\partial_{\delta}+(\alpha_{1}-\alpha_{2})\upsilon \left(\upsilon - 1\right)\partial_{\upsilon}\\
    &\,\,\,\,\,-(\alpha_{2}-\alpha_{1})\delta\partial_{\mu_{E}}+\alpha_{1}k_{I}\partial_{k_{I}}+\alpha_{2}k_{E}\partial_{k_{E}}.
    \end{split}
    \label{eq:X_SEI_app}
\end{equation}
We next find universal parameter invariants of these generators using the method of characteristics.


\subsection{Elucidating structural identifiability by calculating the universal parameter invariants}

The structurally identifiable quantities are given by \textit{universal parameter invariants}. Thus, we need to find parameter invariants that are independent of the arbitrary coefficients $\alpha_{1}$ and $\alpha_{2}$ appearing in Eq.~\eqref{eq:X_SEI_app}. To this end, let $s$ be a parameter that parametrises the parameter invariants $I$ of interest according to
\begin{equation}
I(s)=I(c(s),\beta(s),\delta(s),\upsilon(s),\mu_{S}(s),\mu_{E}(s),\mu_{I}(s),k_{E}(s),k_{I}(s)).
\label{eq:I_SEI}
\end{equation}
Then, the characteristic equations are given by
\begin{align}
  \frac{\mathrm{d}c}{\mathrm{d}s}&=\chi_{c}=(\alpha_{2}-\alpha_{1})c\upsilon-\alpha_{2}c,\label{eq:chara_c}\\
  \frac{\mathrm{d}\beta}{\mathrm{d}s}&=\chi_{b}=\alpha_{1}\beta,\label{eq:chara_beta}\\
  \frac{\mathrm{d}\delta}{\mathrm{d}s}&=\chi_{\delta}=(\alpha_{2}-\alpha_{1})\delta,\label{eq:chara_delta}\\
  \frac{\mathrm{d}\upsilon}{\mathrm{d}s}&=\chi_{\upsilon}=(\alpha_{1}-\alpha_{2})\upsilon \left(\upsilon - 1\right),\label{eq:chara_epsilon}\\
  \frac{\mathrm{d}\mu_{S}}{\mathrm{d}s}&=\chi_{\mu_{S}}=0,\label{eq:chara_muS}\\
  \frac{\mathrm{d}\mu_{E}}{\mathrm{d}s}&=\chi_{\mu_{E}}=-(\alpha_{2}-\alpha_{1})\delta,\label{eq:chara_muE}\\
  \frac{\mathrm{d}\mu_{I}}{\mathrm{d}s}&=\chi_{\mu_{I}}=0,\label{eq:chara_muI}\\
  \frac{\mathrm{d}k_{E}}{\mathrm{d}s}&=\chi_{k_{E}}=\alpha_{2}k_{E},\label{eq:chara_kE}\\
  \frac{\mathrm{d}k_{I}}{\mathrm{d}s}&=\chi_{k_{I}}=\alpha_{1}k_{I}.\label{eq:chara_kI}  
\end{align}
The universal parameter invariants are first integrals of Eqs.~\eqref{eq:chara_c}--\eqref{eq:chara_kI} that are independent of the arbitrary coefficients $\alpha_{1}$ and $\alpha_{2}$. By Eqs.~\eqref{eq:chara_muS} and~\eqref{eq:chara_muI}, two universal parameter invariants are given by
\begin{equation}
    I_{1}=\mu_{I},\quad I_{2}=\mu_{S},
\label{eq:SEI_I_1_and_I_2_app}
\end{equation}
since the corresponding infinitesimals are zero, i.e. $\chi_{\mu_{I}}=\chi_{\mu_{S}}=0$. Thus, $\mu_{I}$ and $\mu_{S}$ are the only two parameters that are directly structurally identifiable. 

All of the remaining rate parameters are unidentifiable, and next we set out to find the remaining universal parameter invariants. Combining Eqs.~\eqref{eq:chara_delta} and~\eqref{eq:chara_muE}, we obtain
\begin{equation}
  \frac{\mathrm{d}\delta}{\mathrm{d}\mu_{E}}=-1,
  \label{eq:chara_I3}
\end{equation}
and the corresponding universal parameter invariant is given by
\begin{equation}
  I_{3}=\delta+\mu_{E}.
  \label{eq:I3_app}
\end{equation}
Combining Eqs.~\eqref{eq:chara_beta} and~\eqref{eq:chara_kI}, we obtain
\begin{equation}
  \frac{\mathrm{d}\beta}{\mathrm{d}k_{I}}=\frac{\beta}{k_{I}},
  \label{eq:chara_I4}
\end{equation}
and the corresponding universal parameter invariant is given by
\begin{equation}
  I_{4}=\frac{\beta}{k_{I}}.
  \label{eq:I4_app}
\end{equation}
Combining Eqs.~\eqref{eq:chara_delta} and~\eqref{eq:chara_epsilon}, we obtain
\begin{equation}
  \frac{\mathrm{d}\upsilon}{\mathrm{d}\delta}=\frac{\upsilon(1-\upsilon)}{\delta},
  \label{eq:chara_I5}
\end{equation}
and the corresponding universal parameter invariant is given by
\begin{equation}
  I_{5}=\frac{\upsilon}{\delta(1-\upsilon)}.
  \label{eq:I5_app}
\end{equation}

These five invariants are simple to calculate as it is obvious how the arbitrary coefficients $\alpha_{1}$ and $\alpha_{2}$ are eliminated. For the two remaining invariants involving the parameters $c$ and $k_{E}$ some algebraic manipulations are required to eliminate the coefficients $\alpha_{1}$ and $\alpha_{2}$ in order to find the corresponding universal parameter invariants. Starting with the parameter $c$, we consider the product $c\upsilon$. Using the characteristic equations in Eqs.~\eqref{eq:chara_c} and~\eqref{eq:chara_epsilon} yields
\begin{equation}
  \frac{\mathrm{d}(c\upsilon)}{\mathrm{d}s}=\upsilon\chi_{c}+c\chi_{\upsilon}=-\alpha_{1}(c\upsilon),
  \label{eq:chara_ceps}
\end{equation}
and combining the resulting characteristic equation for $c\upsilon$ with that in Eq.~\eqref{eq:chara_beta} yields
\begin{equation}
  \frac{\mathrm{d}(c\upsilon)}{\mathrm{d}\beta}=-\frac{(c\upsilon)}{\beta}.
  \label{eq:chara_I6}
\end{equation}
The corresponding universal parameter invariant is therefore given by
\begin{equation}
I_{6}=\beta{c}\upsilon.
\label{eq:I6_app}
\end{equation}
Finally, for the parameter $k_{E}$ we consider the quotient $\delta/k_{E}$. Using the characteristic equations in Eqs.~\eqref{eq:chara_delta} and~\eqref{eq:chara_kE} we obtain the following characteristic equation for the quotient $\delta/k_{E}$
\begin{equation}
\frac{\mathrm{d}}{\mathrm{d}s}\left(\frac{\delta}{k_{E}}\right)=\frac{k_{E}\chi_{\delta}-\delta\chi_{k_{E}}}{k_{E}^{2}}=-\alpha_{1}\left(\frac{\delta}{k_{E}}\right),
\label{eq:chara_deltakE}
\end{equation}
and combining this characteristic equation with that in Eq.~\eqref{eq:chara_beta} yields
\begin{equation}
\frac{\mathrm{d}(\delta/k_{E})}{\mathrm{d}\beta}=-\frac{(\delta/k_{E})}{\beta}.
\label{eq:chara_I7}
\end{equation}
The last universal parameter invariant is therefore
\begin{equation}
I_{7}=\frac{\beta\delta}{k_{E}}.
\label{eq:I7_app}
\end{equation}


\begin{thebibliography}{10}

\bibitem{chis2011structural}
Oana-Teodora Chis, Julio~R Banga, and Eva Balsa-Canto.
\newblock Structural identifiability of systems biology models: a critical
  comparison of methods.
\newblock {\em PlOS ONE}, 6(11):e27755, 2011.

\bibitem{ljung1994global}
Lennart Ljung and Torkel Glad.
\newblock On global identifiability for arbitrary model parametrizations.
\newblock {\em Automatica}, 30(2):265--276, 1994.

\bibitem{hong2020global}
Hoon Hong, Alexey Ovchinnikov, Gleb Pogudin, and Chee Yap.
\newblock Global identifiability of differential models.
\newblock {\em Communications on Pure and Applied Mathematics},
  73(9):1831--1879, 2020.

\bibitem{saccomani2001new}
M~Pia Saccomani, Stefania Audoly, Giuseppina Bellu, and Leontina D'Angio.
\newblock A new differential algebra algorithm to test identifiability of
  nonlinear systems with given initial conditions.
\newblock In {\em Proceedings of the 40th IEEE Conference on Decision and
  Control (Cat. No. 01CH37228)}, volume~4, pages 3108--3113. IEEE, 2001.

\bibitem{walter1982global}
Eric Walter and Yves Lecourtier.
\newblock Global approaches to identifiability testing for linear and nonlinear
  state space models.
\newblock {\em Mathematics and Computers in Simulation}, 24(6):472--482, 1982.

\bibitem{yates2009structural}
James W~T Yates, Neil~D Evans, and Michael~J Chappell.
\newblock {Structural identifiability analysis via symmetries of differential
  equations}.
\newblock {\em Automatica}, 45(11):2585--2591, 2009.

\bibitem{merkt2015higher}
Benjamin Merkt, Jens Timmer, and Daniel Kaschek.
\newblock {Higher-order Lie symmetries in identifiability and predictability
  analysis of dynamic models}.
\newblock {\em Physical Review E}, 92(1):12920, 2015.

\bibitem{massonis2020finding}
Gemma Massonis and Alejandro~F Villaverde.
\newblock {Finding and breaking lie symmetries: implications for structural
  identifiability and observability in biological modelling}.
\newblock {\em Symmetry}, 12(3):469, 2020.

\bibitem{castro2020structuralIdentifiability}
Mario Castro and Rob~J de~Boer.
\newblock {Testing structural identifiability by a simple scaling method}.
\newblock {\em PLOS Computational Biology}, 16(11):1--15, 2020.

\bibitem{villaverde2022symmetries}
Alejandro~F Villaverde.
\newblock {Symmetries in Dynamic Models of Biological Systems: Mathematical
  Foundations and Implications}.
\newblock {\em Symmetry}, 14(3):467, 2022.

\bibitem{villaverde2021testing}
Alejandro~F Villaverde and Gemma Massonis.
\newblock On testing structural identifiability by a simple scaling method:
  relying on scaling symmetries can be misleading.
\newblock {\em PLoS computational biology}, 17(10):e1009032, 2021.

\bibitem{bluman1989symmetries}
George~W. Bluman and Sukeyuki Kumei.
\newblock {\em Symmetries and differential equations}.
\newblock Springer Science {\&} Business Media, New York, 1989.

\bibitem{hydon2000symmetry}
Peter~E. Hydon.
\newblock {\em Symmetry methods for differential equations: a beginner's
  guide}.
\newblock Cambridge University Press, New York, 2000.

\bibitem{olver2000applications}
Peter~J Olver.
\newblock {\em Applications of Lie groups to differential equations}.
\newblock Springer Science {\&} Business Media, New York, 2000.

\bibitem{stephani1989differential}
Hans Stephani.
\newblock {\em Differential equations: their solution using symmetries}.
\newblock Cambridge University Press, New York, 1989.

\bibitem{renardy2022structural}
Marissa Renardy, Denise Kirschner, and Marisa Eisenberg.
\newblock Structural identifiability analysis of age-structured {PDE} epidemic
  models.
\newblock {\em Journal of Mathematical Biology}, 84(1-2):9, 2022.

\bibitem{meshkat2015identifiability}
Nicolette Meshkat, Seth Sullivant, and Marisa Eisenberg.
\newblock Identifiability results for several classes of linear compartment
  models.
\newblock {\em Bulletin of Mathematical Biology}, 77(8):1620--1651, 2015.

\bibitem{meshkat2014repara}
Nicolette Meshkat and Seth Sullivant.
\newblock Identifiable reparametrizations of linear compartment models.
\newblock {\em Journal of Symbolic Computation}, 63:46--67, 2014.

\bibitem{audoly2001global}
Stefania Audoly, Giuseppina Bellu, Leontina D'Angio, Maria~Pia Saccomani, and
  Claudio Cobelli.
\newblock Global identifiability of nonlinear models of biological systems.
\newblock {\em IEEE Transactions on biomedical engineering}, 48(1):55--65,
  2001.

\bibitem{bolie1961coefficients}
Victor~W Bolie.
\newblock Coefficients of normal blood glucose regulation.
\newblock {\em Journal of {A}pplied {P}hysiology}, 16(5):783--788, 1961.

\bibitem{cobelli1980parameter}
Claudio Cobelli and Joseph~J Distefano~3rd.
\newblock Parameter and structural identifiability concepts and ambiguities: a
  critical review and analysis.
\newblock {\em American Journal of Physiology-Regulatory, Integrative and
  Comparative Physiology}, 239(1):R7--R24, 1980.

\bibitem{meurer2017sympy}
Aaron Meurer, Christopher~P. Smith, Mateusz Paprocki, Ond\v{r}ej
  \v{C}ert\'{i}k, Sergey~B. Kirpichev, Matthew Rocklin, AMiT Kumar, Sergiu
  Ivanov, Jason~K. Moore, Sartaj Singh, Thilina Rathnayake, Sean Vig, Brian~E.
  Granger, Richard~P. Muller, Francesco Bonazzi, Harsh Gupta, Shivam Vats,
  Fredrik Johansson, Fabian Pedregosa, Matthew~J. Curry, Andy~R. Terrel,
  \v{S}t\v{e}p\'{a}n Rou\v{c}ka, Ashutosh Saboo, Isuru Fernando, Sumith Kulal,
  Robert Cimrman, and Anthony Scopatz.
\newblock Sympy: symbolic computing in python.
\newblock {\em PeerJ Computer Science}, 3:e103, January 2017.

\bibitem{rey2023benchmarking}
Xabier Rey~Barreiro and Alejandro~F Villaverde.
\newblock Benchmarking tools for a priori identifiability analysis.
\newblock {\em Bioinformatics}, 39(2):btad065, 2023.

\bibitem{dong2023structuralIdentifiabilityJL}
Ruiwen Dong, Christian Goodbrake, Heather~A. Harrington, and Gleb Pogudin.
\newblock Differential elimination for dynamical models via projections with
  applications to structural identifiability.
\newblock {\em SIAM Journal on Applied Algebra and Geometry}, 7(1):194--235,
  2023.

\end{thebibliography}
\end{document}